\documentclass[12pt,reqno]{amsart}
\usepackage{amsmath,amssymb,amscd,graphicx}
\usepackage[all]{xy}
\usepackage[latin1]{inputenc}
\usepackage[T1]{fontenc}
\usepackage{hyperref}
\usepackage{ifthen}
\usepackage{fullpage}

\newcommand{\C}{{\mathbb{C}}}
\newcommand{\R}{{\mathbb{R}}}
\newcommand{\Z}{{\mathbb{Z}}}

\newcommand{\cU}{{\mathcal{U}}}

\newcommand{\CC}{{\mathbb{C}}}  
\newcommand{\NN}{{\mathbb{N}}}  
\newcommand{\PP}{{\mathbb{P}}}  
\newcommand{\RR}{{\mathbb{R}}}  
\renewcommand{\SS}{{\mathbb{S}}} 
\newcommand{\ZZ}{{\mathbb{Z}}}  

\newcommand{\supp}{{\operatorname{supp}}}  

\newcommand{\U}{{\operatorname{U}}}  

\newcommand{\hook}{{\lrcorner\,}} 
\newcommand{\limit}[1]{{\underset{#1}{\lim}}} 
\newcommand{\red}[1]{{/\!/\!_{#1}\,}} 

\newcommand{\del}{\partial}  
\newcommand{\delbar}{\overline{\partial}}  
\newcommand{\eps}{\varepsilon}  
\newcommand{\inv}{^{-1}} 
\newcommand{\cut}{\operatorname{cut}} 

\newtheorem{theorem}{Theorem}[section]
\newtheorem{proposition}[theorem]{Proposition}

\newtheorem{corollary}[theorem]{Corollary}
\newtheorem{lemma}[theorem]{Lemma}

\theoremstyle{definition}
\newtheorem{definition}[theorem]{Definition}

\theoremstyle{remark}
\newtheorem{remark}[theorem]{Remark}

\numberwithin{equation}{section}

\begin{document}

\title{Tame Circle Actions}


\author{Susan Tolman}
\address{Department of Mathematics, University of Illinois at Urbana-Champaign, Urbana, Illinois, USA 61801}
\email{stolman@math.uiuc.edu}

\author{Jordan Watts}
\address{Department of Mathematics, University of Colorado Boulder, Boulder, Colorado, USA 80309}
\email{jordan.watts@colorado.edu}

\thanks{Susan Tolman is partially supported by National Science Foundation
Grant DMS-1206365.
Jordan Watts thanks the University of Illinois at Urbana-Champaign for their support. Moreover, this manuscript was significantly improved by suggestions from an anonymous referee; the authors extend their thanks.}



\date{June 15, 2017}


\begin{abstract}
In this paper, we consider 
Sjamaar's holomorphic slice theorem,
the birational equivalence theorem of Guillemin and Sternberg,
and a number of important standard constructions that 
work for Hamiltonian circle actions in both the symplectic category and the K\"ahler category:
reduction, cutting, and blow-up. 
In each case, we show that the theory extends to Hamiltonian circle
actions on complex manifolds with tamed symplectic forms. (At least,
the theory extends if the fixed points are isolated.)

Our main motivation for this paper is  that the
first author needs the machinery that we develop here to
construct a non-Hamiltonian symplectic circle action on a closed, connected six-dimensional symplectic manifold with exactly 32 fixed points; this answers
an open question in symplectic geometry.
However, we also believe that the setting we work in is intrinsically interesting, and elucidates the key role played by the following fact: the moment image
of $e^t \cdot x$ increases as $t \in \R$ increases.
\end{abstract}

\maketitle

\section{Introduction}\label{s:introduction}

In this paper, we consider 
Sjamaar's holomorphic slice theorem \cite{reyer},
the birational equivalence theorem of Guillemin and Sternberg \cite{GS89},
and a number of important standard constructions that 
work for Hamiltonian circle actions in both the symplectic category and the K\"ahler category:
reduction, cutting \cite{Le}, and blow-up.
In each case, we show that the theory extends to Hamiltonian circle
actions on complex manifolds with tamed symplectic forms. (At least,
the theory extends if the fixed points are isolated.)

Our main motivation for this paper is 
the following question, which appears in McDuff and Salamon \cite{MS},
and is often referred to as the ``McDuff conjecture'': {\it Does there exist a non-Hamiltonian symplectic circle action with isolated fixed points on a closed, connected symplectic manifold?}
The first author needs the machinery that we develop here to
answer this question by constructing a non-Hamiltonian
symplectic circle action on  a closed, connected  six-dimensional symplectic manifold 
with exactly 32 isolated fixed points in \cite{T}.
Propositions~\ref{p:reduction}, \ref{p:prop}, and \ref{p:with crit pts orbifold}
play a key role in ``adding'' the 
fixed points and analysing the resulting manifold in that paper.

Because of this motivation, we focus on the case that the fixed points are isolated, sometimes allow orbifolds with isolated $\Z_2$-singularities,
and  work with a slight generalisation
of tamed forms.  To explain concretely, we  introduce some notation.
Let $\CC^\times $ act holomorphically on a complex manifold $(M,J)$.
Let  $\xi_M$ denote the vector field on $M$ induced by the restricted action
of $\R/2 \pi \Z \cong \SS^1 \subset \CC^\times$.
Let $M^{\SS^1}$ be the set of points fixed by this action, and  let  $\Omega^k(M)^{\SS^1}$ be the set of $\SS^1$-invariant $k$-forms.  
Consider a symplectic form $\omega \in \Omega^2(M)^{\SS^1}$ and assume
that $\Psi \colon M \to \R$ is a moment map, i.e., $\xi_M \hook \omega = - d \Psi$.
Recall that $J$ {\bf tames} $\omega$ if  $\omega( v,J(v)) > 0$ for all non-zero tangent vectors $v$.
In this paper, we also work with the following significantly weaker condition:
We say that the action {\bf tames} $\omega$ if
$\omega(\xi_M,J(\xi_M)) > 0$ on $M \smallsetminus M^{\SS^1}$.

We believe that this setting is intrinsically interesting, and 
hope that it will prove to be a fruitful source of examples and 
counter-examples. 
To see why, 
note that if $\omega \in \Omega^2(M)$ is a K\"ahler form (and $M$ is compact), then the logarithm of
the Duistermaat-Heckman function $\mu$ is a concave function on the moment map
image $\Psi(M)$ \cite{Gr}; in particular,  $\mu$ has no strict local minima on the 
interior  $\Psi(M)^\circ$.
If $\omega$ is  tamed by $J$ instead,
then the Duistermaat-Heckman function
need not be log-concave; nevertheless, if $\dim_\R M = 6$, then
$\mu$ cannot have a strict local minimum at certain $a \in \R$, e.g.,
if $\SS^1 \subset \CC^\times$ acts freely on $\Psi\inv(a)$.
In contrast, if $\omega$ is merely tamed by the action, 
then the Duistermaat-Heckman function  can have strict local minima at such values.
Thus, for some of the key pieces that the first author  used to construct the non-Hamiltonian example, 
the symplectic form was tamed by the action but was not (and could not be)
K\"ahler, or even tamed by $J$ \cite{T}.


Since $\omega( \cdot, J (\cdot))$ may no longer be a metric  for tame symplectic forms,
some of the proofs that work for K\"ahler manifolds
do not work for this larger class without significant modification.
However, 
in the K\"ahler case the gradient flow $\nabla \Psi$ is equal to $-J(\xi_M)$, which is the vector field induced by the $\R$-action
given by $(t,x) \mapsto e^t \cdot x$ for all $t\in \R$ and $x \in M$.
Hence,  the function $t \mapsto \Psi(e^t \cdot x)$ is  increasing.
This  remains true in our setting; we will use it
repeatedly throughout this paper.
\begin{lemma}\label{l:monotone}
Let $(M,J)$ be a complex manifold with a holomorphic $\CC^\times$-action, a symplectic form $\omega\in\Omega^2(M)^{\SS^1}$ tamed by the action, and a moment map $\Psi\colon M\to\RR$.  
Then the function $t\mapsto\Psi(e^t\cdot x)$ is strictly  increasing for  all $x\in M\smallsetminus M^{\SS^1}$.
\end{lemma}

\begin{proof}
Given $x\in M\smallsetminus M^{\SS^1}$, let $\gamma_x(t):=\Psi(e^t\cdot x)$.  Then for all $t\in\RR$, $$\dot\gamma_x(t)=d\Psi(-J(\xi_M))\Big|_{e^t\cdot x}=\omega(\xi_M,J(\xi_M))\Big|_{e^t\cdot x}>0.$$
\end{proof}

\section{Tame Local Normal Form}\label{s:local normal form}


The goal of this section is to prove a $\CC^\times$-equivariant holomorphic Bochner linearisation theorem in our setting of a symplectic structure tamed by the action.
  (More precisely, we develop a $\CC^\times$-equivariant holomorphic local normal form for a neighbourhood of a finite number of fixed points in the same moment fibre.)  Our proof is adapted from Sjamaar's proof of the  holomorphic slice theorem for K\"ahler actions \cite[Section 1]{reyer}; see also \cite{HL}.

\begin{proposition}\label{p:local normal form}
Let $(M,J)$ be a complex manifold with a holomorphic $\CC^\times$-action, a symplectic form $\omega\in\Omega^2(M)^{\SS^1}$ tamed by the action, and a moment map $\Psi\colon M\to\RR$. Given $\{p_1,\dots,p_k\} \in M^{\SS^1} \cap \Psi\inv(0)$,
there exists a $\CC^\times$-invariant neighbourhood  of $\{p_1,\dots,p_k\}$ in $M$ which is $\CC^\times$-equivariantly biholomorphic to a neighbourhood of $\coprod_{j=1}^k\{0\}$ in $\coprod_{j=1}^k\CC^n$, where $\CC^\times$ acts linearly on each $\CC^n$.
\end{proposition}

To prove this, we need a holomorphic version of the 
Bochner linearisation theorem 
for compact group actions on complex manifolds.  While this is well-known (\emph{cf.}, for example, \cite{GGK}), 
we include the proof for completeness.  Since we consider actions on orbifolds in Proposition~\ref{p:blowup2}, we will prove the orbifold version.  Here, if a group $G$ acts on a complex orbifold $(M,J)$, we say that the action \textbf{respects $J$} if every $g \in G$ induces an automorphism of $(M,J)$.

\begin{lemma}\label{l:bochner}
Let a compact Lie group $G$ act on a complex orbifold $(M,J)$; 
assume that the action respects $J$. 
Given $p\in M^G$, there exists a $G$-equivariant biholomorphism from a
$G$-invariant neighbourhood of $0\in T_pM$ to a neighbourhood of $p\in M$.
\end{lemma}

\begin{proof}
Let $\Gamma$ be the (orbifold) isotropy group of $p$.  
There exists an extension $\widetilde{G}$ of $G$ by $\Gamma$, an action of $\widetilde{G}$ on an open set $\widetilde{U}\subset\CC^n$ 
that respects the complex structure and fixes $\widetilde{p}\in\widetilde{U}$,
 and a $G$-equivariant biholomorphism from $\widetilde{U}/\Gamma$ to $M$ that sends $[\widetilde{p}]$ to $p$. (See, for example, \cite{LT}.)
  Let $\psi:\widetilde{U}\to T_{\widetilde{p}}\widetilde{U}$ be any holomorphic map whose differential at $\widetilde{p}$ is the identity map on $T_{\widetilde{p}}\widetilde{U}$.  Since $\widetilde{G}$ is compact, we can average $\psi$ to obtain a $\widetilde{G}$-equivariant map $\bar{\psi} \colon \widetilde{U}\to T_{\widetilde{p}}\widetilde{U}$ such that $d\bar{\psi}|_{\widetilde{p}}$ is equal to the identity, defined by $$\bar{\psi}(q):=\int_{\widetilde{G}}g_*\psi(g^{-1}\cdot q)dg$$ for all $q\in \widetilde{U}$, where $dg$ is the Haar measure on $\widetilde{G}$.  Since the action respects the complex structure, the map $\bar{\psi}$ is holomorphic.  By the inverse function theorem, we can invert $\bar{\psi}$ on a neighbourhood of $\widetilde{p}$ to construct the required biholomorphism.
\end{proof}

Let $(M_1,J_1)$ and $(M_2,J_2)$ be complex manifolds with holomorphic $\CC^\times$-actions, $A$ be an $\SS^1$-invariant  open subset of $M_1$,
and $\varphi \colon A\to M_2$ be an $\SS^1$-equivariant holomorphic map.
Then $\varphi$ sends the vector  
$J_1(\xi_{M_1})\big|_{x}$ to the vector  $J_2(\xi_{M_2})\big|_{\varphi(x)}$ for all $x \in A$.
Hence, if $(t_-,t_+)$ is the connected component of $\{t \in \R \mid e^t \cdot x \in A \}$ containing $0$, then 
\begin{equation}\label{e:R action}
\varphi(e^t\cdot x)=e^t\cdot\varphi(x) \mbox{ for all } t \in (t_-,t_+).
\end{equation}
However, \eqref{e:R action} need not hold for all $t \in \R$ such that $e^t\cdot x\in A$.  This motivates the following definition.

\begin{definition}\label{d:orb convx}
Let $(M,J)$ be a complex manifold with a holomorphic $\CC^\times$-action. 
Then a subset $A\subseteq M$ is \textbf{orbitally convex} 
(with respect to the $\CC^\times$-action) if it is $\SS^1$-invariant
 and the set $\{ t \in \R  \mid e^{t} \cdot x \in A \}$ is connected for all $x \in A$.
\end{definition}

The following proposition, which is Proposition 1.4 in \cite{reyer}, is a consequence of the discussion above.

\begin{proposition}\label{p:loc to glob action}
Let $(M_1,J_1)$ and $(M_2,J_2)$ be complex manifolds with holomorphic $\CC^\times$-actions.  Assume that $A\subseteq M_1$ is an orbitally convex open set, 
and $\varphi \colon A\to M_2$ is an $\SS^1$-equivariant holomorphic map.  Then, $\varphi$ extends to a $\CC^\times$-equivariant holomorphic map $\widetilde{\varphi} \colon \CC^\times\cdot A\to M_2$. 
Consequently, if $\varphi(A)$ is open and orbitally convex in $M_2$ and $\varphi \colon A\to \varphi(A)$ is biholomorphic, then $\widetilde{\varphi}$ is biholomorphic onto the open set $\CC^\times\cdot \varphi(A)$.
\end{proposition}

\begin{remark}
The proof of Proposition~\ref{p:loc to glob action} in \cite{reyer}
uses the definition stated here in Definition~\ref{d:orb convx}. 
Note, however, that $A \subset M$ may not be orbitally convex even if
the intersection $\{ e^{t} \cdot x \mid t \in \R \}  \cap A$ is connected for all $x \in A$.
To see this,
let $1 \in \Z$ act on $\C^\times \subset \C$ by $z \mapsto 2 z$,
and let $M := \C^\times/\Z$ with the natural $\C^\times$-action.
Let $B \subset M$ be the image of  $\SS^1 \subset \C^\times$, and let
$A =  M \smallsetminus B$.
Then $\{ e^{t} \cdot x \mid t \in \RR \} \cap A$ is connected for all $x \in M$,
but $A$ is not orbitally convex.
\end{remark}

To complete the proof of Proposition~\ref{p:local normal form}, 
we  show that we can choose the neighbourhoods
in Lemma~\ref{l:bochner} to be orbitally convex.  For this, we need the following lemma.

\begin{lemma}\label{l:vectorspace2}
Let $\CC^\times$ act linearly on $\CC^{n}$,
and let $J$ be the standard complex structure on $\CC^n$.
Let $U$ be an $\SS^1$-invariant neighbourhood of $0$ in $\CC^n$, 
with a symplectic form $\omega \in \Omega^2(U)^{\SS^1}$ tamed by the action, 
and a moment map $\Psi \colon U \to \RR$ sending $0$ to $0$.  There exist $\delta>0$ and an orbitally convex neighbourhood $V \subset U$ of $0$ 
so that the following holds:
Given $z \in V$, if  $(t_-,t_+ ) = \{ t \in \R \mid e^t \cdot z \in V \}$ then
\begin{enumerate} 
\item either $t_+ = \infty$, or there exists $t \in (t_-,t_+)$ with $\Psi(e^t \cdot z) > \delta$; and
\item either $t_- = -\infty$  or there exists $t \in (t_-,t_+)$ with $\Psi(e^t \cdot z) < -\delta$.
\end{enumerate}
\end{lemma}

\begin{proof}
We may assume $\CC^n = \CC^{k_-} \times \CC^{k_+} \times \CC^l$, where $\CC^\times$ acts on $\CC^{k_-}$ with negative weights $(-\alpha^-_1,\dots,-\alpha^-_{k_-})\in(-\NN)^{k_-}$, that is, $\lambda\cdot(z_1,\dots,z_{k_-})=(\lambda^{-\alpha_1^-}z_1,\dots,\lambda^{-\alpha_{k_-}^-}z_{k_-})$;
on $\CC^{k_+}$ with positive weights $(\alpha^+_1,\dots,\alpha^+_{k_+}) \in \NN^{k_+}$;
and on $\CC^l$ trivially.  

Define continuous $\SS^1$-invariant functions $N_\pm \colon \CC^{k_\pm}\to\RR$  by $$N_\pm(z)=\bigg(\sum_{j=1}^{k_\pm}|z_j|^{{2}/{\alpha^\pm_j}} \bigg)^\frac{1}{2},$$ and define $N \colon \C^{k_-} \times \C^{k_+}\times \CC^l \to \R$ by
$N(z_-,z_+,w) = N_-(z_-) N_+(z_+).$
Note that 
\begin{gather}
\label{e:neighbourhoods}
N_\pm(e^t\cdot z_\pm)=e^{\pm t} N_\pm(z_\pm),  \ \mbox{and so } \\
\label{e:neighbourhoods2}
N(e^t \cdot (z_-,z_+,w)) =  N(z_-,z_+,w)
\end{gather}
for all $z_\pm \in\CC^{k_\pm}$, $w\in\CC^l$, and $t\in \RR$.

Given $\eps>0$, define 
$$D^\pm_\eps=\{ z \in \CC^{k_\pm} \mid N_\pm(z) < \eps \} 
\mbox{  and } 
S^\pm_\eps = \{z \in \CC^{k_\pm} \mid N_\pm(z) = \eps \} 
\subsetneq \overline{D^\pm_\eps}.$$  
If $\eps<1$, then $|z_j|<1$ for all $z \in D^\pm_\eps$ and  $j \in \{1, \dots, k_\pm\}$, 
and so  $N_\pm(z) \geq \big(\sum_j|z_j|^2\big)^\frac{1}{2}=|z|.$ 
Thus there exists $\eps > 0$ and a compact, connected  neighbourhood $K$ of $0 \in \CC^l$ such that $\overline{D^-_\eps} \times \overline{D^+_\eps}\times K\subset U$.  

If $z\in \overline{D^\pm_\eps}$, then $\limit{t\to \mp\infty}e^t\cdot z=0.$  Moreover, $e^t \cdot z \in D^\pm_\eps$ for all $t \in \mp(0,\infty)$ 
by \eqref{e:neighbourhoods}.  
Additionally, $\Psi(0,0,w)=0$ for any $w\in K$; therefore,  
by Lemma~\ref{l:monotone}, $\Psi(z_-,0,w)<0$ for every non-zero $z_-\in \overline{D^-_\eps}$, 
and $\Psi(0,z_+,w)>0$ for every non-zero $z_+\in \overline{D^+_\eps}$.
  Hence, since 
$S^\pm_\eps$ and $K$ are compact, there exist $\delta>0$ and $\eps'\in(0,\eps)$ such that 
\begin{equation}\label{e:neighbourhoods3}
\Psi(S^-_\eps\times D^+_{\eps'}\times K)\subsetneq(-\infty,-\delta) \text{  and  } \Psi(D^-_{\eps'}\times S^+_\eps\times K)\subsetneq(\delta,\infty).
\end{equation}  

Define $V:=\{(z_-,z_+,w)\in D^-_\eps\times D^+_\eps\times K\mid N(z_-,z_+) <\eps\eps'\} \subset U$.  
Then $V$ is an orbitally convex neighbourhood of $0$ by \eqref{e:neighbourhoods} and \eqref{e:neighbourhoods2}.
Fix $(z_-,z_+,w) \in V$.
If $z_+ = 0$, then $N(e^t \cdot (z_-,z_+,w)) = 0$ for all $t \in \R$,
and so  $e^t \cdot (z_-, z_+,w)$ lies in $V$ for all $t \geq 0$.
On the other hand, if $z_+ \neq 0$, then  $t_+ := \ln \frac{\eps}{N_+(z_+)} > 0$
because $N_+(z_+) <\eps$. 
 A straightforward calculation using \eqref{e:neighbourhoods} shows that $e^t\cdot (z_-,z_+,w)\in D^-_\eps\times D^+_\eps\times K$ for all  $t\in[0,t_+)$.  
By \eqref{e:neighbourhoods2}, this implies that $e^t\cdot (z_-,z_+,w)\in V$ for all $t\in[0,t_+)$. 
Since $e^{t_+}\cdot (z_-,z_+,w)\in D^-_{\eps'}\times S^+_\eps\times K$, Claim 1 follows from \eqref{e:neighbourhoods3}. 
The proof of Claim 2 is nearly identical.
\end{proof}

\begin{proof}[Proof of Proposition~\ref{p:local normal form}]
By Lemma~\ref{l:bochner}, there exists an $\SS^1$-equivariant 
biholomorphism $\varphi$ from an $\SS^1$-invariant neighbourhood $U$ of 
$\coprod_j \{0\}$ in $\coprod_j \CC^n$ to a neighbourhood of $\{p_1,\dots,p_k\}$ in $M$, where $\CC^\times$ acts linearly on each $\CC^n$. 
Let $\delta$ be the positive real and
$V \subset U$ be the orbitally convex neighbourhood of $\coprod_j \{0\}$
given by repeatedly applying Lemma~\ref{l:vectorspace2}.
Fix $x \in \varphi(V)$.  Assume that $\varphi(z) = e^{t'} \cdot x$,
where $z \in V$ and $t' \in \R$,
and let $(t_-,t_+) :=\{t \in \RR \mid e^t\cdot z \in V \}$. 
Then, since $V$ is orbitally convex,
 Equation \eqref{e:R action} (alternatively,  Proposition~\ref{p:loc to glob action}) 
implies that $\varphi(e^t \cdot z) = e^{t+t'} \cdot x$ 
for all $t \in (t_-,t_+)$. 
In particular, $e^{t} \cdot x \in \varphi(V)$ for all $t \in (t' +t_-,t' +t_+)$.
Moreover, by Lemma~\ref{l:vectorspace2},
\begin{enumerate} 
\item either $t_+=\infty$, or there exists $t\in(t' + t_-,t' +t_+)$ 
with $\Psi(e^{t}\cdot x) > \delta$; and
\item either $t_-=-\infty$, or there exists $t\in(t' +t_-,t'+ t_+)$ with 
$\Psi(e^{t}\cdot x) < -\delta$.
\end{enumerate}
Finally, by Lemma~\ref{l:monotone}, the function 
$t \mapsto \Psi(e^t  \cdot x)$ is 
increasing.  
Therefore, $\varphi(V)$ is open and orbitally convex. 
Thus, by Proposition~\ref{p:loc to glob action}, $\varphi|_V$ extends to a $\CC^\times$-equivariant biholomorphism from $\CC^\times\cdot V$ to $\CC^\times\cdot \varphi(V)$.
\end{proof}

Since the function $t \mapsto \Psi(e^t \cdot x)$ is increasing,
Claims (1) and (2) above imply that
 $\Psi(e^t \cdot x) \not\in (-\delta, \delta)$ if  $t \not\in(t' + t_-, t'+ t_+)$. Since also $e^t \cdot x \in \varphi(V)$ for all $t \in ( t'+ t_-, t' + t_+)$,
the above proof has the following corollary.

\begin{corollary}\label{delta}
Let $(M,J)$ be a complex manifold with a holomorphic $\CC^\times$-action, 
a symplectic form $\omega\in\Omega^2(M)^{\SS^1}$ tamed by the action, and a moment map $\Psi\colon M\to\RR$. 
Given $p \in  M^{\SS^1} \cap \Psi\inv(0)$ and an $\SS^1$-invariant neighbourhood $U$
of $p$, there exists $\delta > 0$ and
an $\SS^1$-invariant open neighbourhood $V \subseteq U$ of $p$
such that  $\C^\times \cdot V \cap \Psi\inv(-\delta,\delta) \subset V$.
\end{corollary}

\section{Tame Reduction}\label{s:reduced spaces}

It is well-known that the reduced spaces of K\"ahler manifolds naturally
inherit K\"ahler structures \cite{GS82,Ki,HKLR}.
The goal of this section is to generalise this fact to our setting; in particular,  the reduced spaces of complex
manifolds with  tamed symplectic forms  naturally inherit complex structures
that  tame the reduced symplectic form.

\begin{proposition}\label{p:reduction}
Let $(M,J)$ be a complex manifold with a holomorphic $\CC^\times$-action, a symplectic form $\omega\in\Omega^2(M)^{\SS^1}$ tamed by the action, and a moment map $\Psi\colon M\to\RR$. 
Given a regular value $a$  of $\Psi$, 
let $U_a:=\CC^\times\cdot\Psi^{-1}(a)$. 
Then the following hold.
\begin{enumerate}
\item \label{i:complex reduction}
 The quotient $U_a/\CC^\times$ is naturally a complex orbifold,  and $U_a$ is a holomorphic $\CC^\times$-bundle over $U_a/\CC^\times$.
\item \label{i:biholom} There is a complex structure $J_a$ on the reduced space $M\red{a}\SS^1$ 
so that the inclusion $\Psi^{-1}(a) \hookrightarrow U_a$ induces a biholomorphism $M\red{a}\SS^1\to U_a/\CC^\times$.
\item For all $x \in \Psi\inv(a)$, the natural map
$$ \{ X \in T_x M \mid \omega(\xi_M,X) = \omega(\xi_M,J(X)) = 0\}\to T_{[x]} (M \red{a}\SS^1)$$
is an isomorphism of complex and symplectic vector spaces.
Consequently, if $J$ tames $\omega$ at $x$ then $J_a$ tames the reduced symplectic form at $[x]$.
\end{enumerate}
\end{proposition}

\begin{remark}\label{r:kaehlar}
In the situation of Proposition~\ref{p:reduction}, if $(M,J,\omega)$
is K\"ahler then the reduced space $M \red{a} \SS^1$ is K\"ahler by Claim (3).
Thus, the standard  theorem for K\"ahler manifolds is a special case.
We have written our proofs to show that the
analogous statement holds whenever applicable in this paper:
cutting (Proposition~\ref{p:cutting}), blow-ups (Proposition~\ref{p:blowup2-mfld} and
Proposition~\ref{p:blowup2}),
and adding fixed points  (Proposition~\ref{p:prop}).
Here, we rely on the final
claims in Lemmas~\ref{l:modifyforms1b}, \ref{l:modifyforms2},
and Lemma~\ref{l:modifyforms}.
\end{remark}

\begin{remark}\label{r:toricreduction}
A {\bf symplectic toric orbifold}
is a triple $(M,\omega,\Phi)$, 
where $(M,\omega)$ is a $2n$-dimensional compact, connected
symplectic orbifold and $\Phi \colon M \to \R^n$ is a moment
map for an effective $(\SS^1)^n$-action.
The {\bf moment polytope} $\Delta := \Phi(M)$ is a convex rational simple
polytope.
Given a facet $F \subset \Delta$ with interior $F^\circ$,
there exists a natural number $k_F$ 
so that
$\Z_{k_F}$
is the orbifold isotropy subgroup of 
each point in $\Phi\inv(F^\circ)$.
Symplectic toric orbifolds are
classified (up to $(\SS^1)^n$-equivariant symplectomorphism)
by their moment polytopes (up to translation) and these
natural numbers.
Moreover, the stabiliser of $x \in M$ is the connected subgroup
$H \subset (\SS^1)^n$  with Lie algebra $\mathfrak{h}$, where $\Phi(x) + \mathfrak{h}^0$
is the minimal affine plane that contains  
a face of $\Delta$ containing $\Phi(x)$; in particular, $x \in  M^{(\SS^1)^n}$ exactly if $\Phi(x) \in \Delta$ is a vertex.

Hence, we can visualise  the  constructions in this paper
by considering the case that the circle
$\SS^1 \times \{1\}^{n-1} \subset (\SS^1)^n$  acts on a toric manifold
$(M,\omega,\Phi)$,
and so  the  $\SS^1$-moment map $\Psi$ is the first component of $\Phi$.
Since all symplectic manifolds with Hamiltonian circle actions
are locally isomorphic to toric manifolds,
this gives valuable insight into the underlying symplectic geometry in the general case.

For example, 
in the situation described above,  $a \in \R$ is a regular value of $\Psi$ exactly if no vertex of $\Delta$ lies in $\{a\} \times \R^{n-1}$.
In this case, the reduced space 
$M\red{a}\SS^1$
is a symplectic toric orbifold with
moment polytope  $\Delta \cap \big( \{a\} \times \R^{n-1}\big).$
See Remarks~\ref{r:toriccutting}, \ref{r:toricblowup}, \ref{r:toricfixed}, and
\ref{r:toricbirational} for further discussion.
\end{remark}

To prove Proposition~\ref{p:reduction}, we will need the following important technical lemma,
which only depends on Lemma~\ref{l:monotone} and the fact that $\Psi$
is equivariant.

\begin{lemma}\label{l:monotone2}
Let $(M,J)$ be a complex manifold with a holomorphic $\CC^\times$-action, 
a symplectic form $\omega\in\Omega^2(M)^{\SS^1}$ tamed by the action, 
and a  moment map $\Psi\colon M\to\RR$. 
Define 
$$\cU:=\{(x,s) \in \big(M\smallsetminus M^{\SS^1}\big) \times \R \mid
s \in \Psi(\CC^\times\cdot x)\}.$$  
Then for all $(x, s) \in \cU$, there exists a unique $f(x,s) \in \RR$ such that $$\Psi(e^{f(x,s)}\cdot x)=s;$$
moreover, $\cU$ is open  and $f \colon \cU \to \RR$ is a smooth $\SS^1$-invariant function.
\end{lemma}

\begin{proof}
Given $(x,s) \in \cU$, let $\gamma_x(t):=\Psi(e^t\cdot x)$ for all $t \in \R$.
The moment map
is $\SS^1$-invariant; hence, by definition, there exists $f(x,s) \in \RR$ such that
$\gamma_x(f(x,s)) = s$.  Moreover, by Lemma~\ref{l:monotone}, $\gamma_x$ is strictly increasing, and so $f(x,s)$ is unique. 
 By the implicit function theorem, there exist open neighbourhoods $V$ 
of $(x,s) \in \cU$ and $W$ of $f(x,s) \in \RR$, and a smooth function 
$\widetilde{f}\colon V\to W$ such that $$\big\{\big((y,u),\widetilde{f}(y,u)\big)~\big|~ (y,u)\in V\big\}=\{(y,u,t)\in V\times W\mid \Psi(e^t\cdot y)=u\}.$$  
Therefore $\cU$ is open and $f$ is smooth.  
Finally, since $\Psi$ is $\SS^1$-invariant, the function $f$ is as well. 
 (Here, $\SS^1$ acts  trivially on $\RR$.)
\end{proof}

\begin{proof}[Proof of Proposition~\ref{p:reduction}]
By Lemma~\ref{l:monotone2},  the $\CC^\times$-invariant set $$U_a := \CC^\times \cdot \Psi\inv(a)= 
\{ x \in M \smallsetminus M^{\SS^1} \mid a \in \Psi(\CC^\times \cdot x) \}$$ is open and there exists a smooth function 
$f\colon U_a\to\RR$ such that $$\Psi(e^{f(x)}\cdot x)=a$$
for all $x \in U_a$.

Define $\Theta \colon \CC^\times \times M \to M \times M$ by
$\Theta( e^{u +iv}, x) = (e^{u + iv} \cdot x, x)$ for all $u,v \in \R$ and $x \in M$.
  Given  a compact set $L \subseteq M \times M$, there exists a closed interval $[s,t] \subset \RR$ so that $f(x_1) \in [s,t]$ and $f(x_2)\in [s,t]$ for all $(x_1,x_2) \in L$.
Since $\Psi$ is $\SS^1$-invariant,  $$f( e^{u+iv} \cdot x) = f(x) - u $$ for all $u,v\in\RR$ and $x\in U_a$.
Thus,  for all $(e^{u+iv},x) \in \Theta^{-1}(L)$, we have $f(x) - u \in [s,t]$ and $f(x) \in [s,t]$,  and so $u \in [s-t, t-s]$.
Thus $\Theta^{-1}(L)$ is contained in a compact set.  Since $\Theta$ is continuous, this implies that $\Theta^{-1}(L)$ is compact, that is,
the $\CC^\times$-action on $U_a$ is proper.  

Since the action is proper, the stabiliser $\Gamma$ of $x$ is finite for all $x \in U_a$.
Moreover, by the slice theorem there is a $\CC^\times$-invariant neighbourhood  
of the orbit $\CC^\times\cdot x$ that is $\CC^\times$-equivariantly biholomorphic to the associated bundle 
$\CC^\times\times_{\Gamma}D$,
where $D$ is a disc in the normal space to  $\C^\times \cdot x$ at $x$.
The quotient map $D\to D/\Gamma$ is an orbifold chart on $U_a/\C^\times$ near $[x]$;
hence, $U_a/\C^\times$ is a complex orbifold and
the quotient map $U_a \to U_a/\C^\times$ is holomorphic.
Since the action of $\Gamma$ on $D$ lifts to
the the diagonal $\Gamma$-action on the product $\CC^\times\times D$,
this implies that $U_a\to U_a/\CC^\times$ is a holomorphic $\CC^\times$-bundle; \emph{cf}. \cite[Corollaries B.31 and B.32]{GGK}.

The natural inclusion map $\Psi^{-1}(a) \hookrightarrow U_a$ descends to a 
well-defined smooth map $i\colon M\red{a}\SS^1 \to U_a/\CC^\times $.
 Similarly, the map $U_a\to\Psi^{-1}(a)$ defined by $x \mapsto e^{f(x)}\cdot x$  descends to a smooth map $g\colon U_a/\CC^\times  \to M\red{a}\SS^1$.  Moreover, these induced maps are inverses of each other.  Under the resulting identification, the symplectic quotient  $M\red{a}\SS^1$ inherits a natural complex structure from $U_a/\CC^\times $.

Fix $x\in\Psi^{-1}(a)$.  By the preceding paragraph, the natural map 
$T_{[x]}(M\red{a}\SS^1)\to T_{[x]}(U_a/\CC^\times)$ is an isomorphism of complex vector spaces. 
 Since $\omega(\xi_M,J(\xi_M))>0$ on $\Psi^{-1}(a)$, 
we can represent every vector  in $T_{[x]}(M\red{a}\SS^1)$ by a unique vector 
in the $J$-invariant subspace $\{X\in T_xM \mid \omega(\xi_M,X)=\omega(\xi_M,J(X))=0\}$ of $T_x M$.
The third claim follows.
\end{proof}

\section{Tame Cutting}\label{s:cutting}

Next, we show that symplectic cutting, developed by Lerman in \cite{Le},
also  works in our setting; see also \cite{BGL} for a discussion of the K\"ahler case.

\begin{proposition}\label{p:cutting}
Let $(M,J)$ be a complex manifold with a holomorphic $\CC^\times$-action, a symplectic form $\omega\in\Omega^2(M)^{\SS^1}$ 
tamed by the action, 
and a  moment map $\Psi\colon M\to\RR$. 
Assume that $0\in\RR$ is a regular value of $\Psi$. 
Then there exists a complex orbifold $(M_{\cut},J_{\cut})$ with a holomorphic $\CC^\times$-action, a symplectic form $\omega_{\cut}\in\Omega^2(M_{\cut})^{\SS^1}$ 
tamed by the action, and a moment map $\Psi_{\cut}\colon M_{\cut}\to\RR$ so that the following hold.
\begin{enumerate}
\item \label{i:image} $\Psi_{\cut}(M_{\cut}) \subseteq (-\infty,0]$.
\item \label{i:zero set} 
A neighbourhood of the fixed component $\Psi_{\text{cut}}^{-1}(0)$
is $\CC^\times$-equivariantly biholomorphic 
to the holomorphic line bundle 
$U_0\times_{\CC^\times}\CC$,
where $U_0 := \CC^\times \cdot \Psi\inv(0)$,  $\CC^\times$
acts diagonally on $U_0 \times \CC$, and $\CC^\times$ acts on $U_0 \times_{\CC^\times} \CC$ by
$\lambda \cdot [x,z] = [\lambda \cdot x,z]$.
\item \label{i:complement} There exists an $\SS^1$-equivariant symplectomorphism from $\Psi^{-1}(-\infty,0)$ to $\Psi_{\cut}^{-1}(-\infty,0)$ 
that induces a biholomorphism between 
the reduced spaces at all regular $s \in (-\infty,0)$.
\item \label{i:cut taming}  If $J$ tames $\omega$ near $\Psi^{-1}(s)$ for some $s \in \R$,
 then $J_{\cut}$ tames $\omega_{\cut}$ near $\Psi_{\cut}^{-1}(s)$.
\item If $\Psi$ is proper, then $\Psi_{\cut}$ is proper.
\end{enumerate}
\end{proposition}

\begin{proof}
Let $J'$ be the product complex structure on $M'=M\times\CC$. 
The diagonal $\CC^\times$-action on $M'$ is holomorphic, 
where $\CC^\times$ acts on $\CC$ by multiplication.
Consider the symplectic form $\omega'=\omega+i dz\wedge d\overline{z}/2
\in\Omega^2(M')^{\SS^1}$. 
If $(x,z)\in M' \smallsetminus (M')^{\SS^1}$, then either $x\not\in M^{\SS^1}$  or $z\neq 0$. 
In either case,  the fact that $\xi_{M'} = \xi_M + \xi_\C$ implies that
$$\omega'(\xi_{M'},J'(\xi_{M'}))\big|_{(x,z)}=\omega(\xi_M,J(\xi_M))\big|_x+|z|^2>0,$$  
and so $\omega'$ is tamed by the action.
The function $\Psi'\colon M'\to\RR$ sending $(x,z)$ to $\Psi(x)+|z|^2/2$ is a moment map.
Since $0$ is a regular value of $\Psi$, it is also a regular value of $\Psi'$.

Define $U'_0:=\CC^\times\cdot(\Psi')^{-1}(0) \subseteq  M'$.
By Proposition \ref{p:reduction}, 
the quotient $U'_0/\CC^\times$ is naturally a complex orbifold,
and $U'_0$ is a holomorphic $\CC^\times$-bundle over $U'_0/\CC^\times$.
Moreover, the reduced space  
\begin{equation}\label{Mcut}
M_{\cut}:=M'\red{0}\SS^1  = \big\{ (x,z) \in M \times \CC \, \big|\, \Psi(x) + |z|^2/2 = 0\big\}/\SS^1
\end{equation}
inherits a symplectic structure $\omega_{\cut}$, and a
complex structure $J_{\cut}$ so that the inclusion $(\Psi')^{-1}(0) \hookrightarrow U_0'$ induces a biholomorphism $M_{\cut}\to U'_0/\CC^\times$. 
Finally, for all $(x,z) \in (\Psi')^{-1}(0)$, the natural map
\begin{equation}\label{cutiso}
\{ X' \in T_{(x,z)} M' \mid \omega'(\xi_{M'},X') = \omega'(\xi_{M'},J(X')) = 0 \}
\to T_{[x,z]}( M_{\cut})
\end{equation}
is an isomorphism of complex and symplectic vector spaces.

Since the holomorphic $\CC^\times$-action on $M'$ given by $\lambda\cdot(x,z)=( \lambda \cdot x,z)$
commutes with the diagonal action, 
it descends to a holomorphic $\CC^\times$-action on $M_{\cut}$. 
Moreover,  $\omega'$ is invariant under the associated $\SS^1$-action
on $M'$, and so
$\omega_{\cut} \in \Omega^2(M_{\cut})^{\SS^1}$. 
The function $\Psi_{\cut}\colon M_{\cut}\to\RR$ defined by
\begin{equation}\label{Psicut}
\Psi_{\cut}([x,z]) = \Psi(x)
\end{equation}
is a moment map for the $\SS^1$-action on $M_{\cut}$.
The vector field $\Xi$ on $M' \smallsetminus (M')^{\SS^1}$ given by
\begin{equation}\label{e:X}
\Xi:=\xi_M - \frac{\omega(\xi_M,J(\xi_M))}{\omega(\xi_M,J(\xi_M))+|z|^2}\, \xi_{M'}
=\frac{|z|^2\xi_M-\omega(\xi_M,J(\xi_M))\xi_{\CC}}{\omega(\xi_M,J(\xi_M))+|z|^2}
\end{equation}
descends to $\xi_{M_{\cut}}$ on $M_{\cut}$ and satisfies
$\omega'(\xi_{M'},\Xi)=\omega'(\xi_{M'},J'(\Xi))=0$.
Since \eqref{cutiso} is an isomorphism of symplectic and complex vector spaces, this implies that
$$\omega_{\cut}(\xi_{M_{\cut}},J_{\cut}(\xi_{M_{\cut}}))
=\omega'(\Xi,J'(\Xi))=\frac{|z|^2\omega(\xi_M,J(\xi_M))}{\omega(\xi_M,J(\xi_M))+|z|^2}.$$  
In particular, this is positive 
if $[x,z] \in M_{\cut} \smallsetminus M_{\cut}^{\SS^1}$,
because then $x \not\in M^{\SS^1}$ and $z \neq 0$;
hence $\omega_{\cut}$ is tamed by the action.  

Claims (1) and (5) are immediate consequences of \eqref{Mcut} and \eqref{Psicut}.

Fix $(x,z)$ in $U_0 \times \C$.
By definition, there exists $t \in \R$ such that $e^t \cdot x \in \Psi\inv(0)$.
Thus, since $0$ is a regular value, Lemma~\ref{l:monotone} implies that the
function $s \mapsto \Psi(e^s \cdot x)$ is strictly increasing.
Therefore,
$$\lim_{s \to - \infty} \Psi' (e^s \cdot (x,z)) = \lim_{s \to -\infty} \Psi(e^s \cdot x) < 0 \quad \mbox{and} \quad \lim_{s \to \infty} \Psi'(e^s \cdot (x,z))
\geq \lim_{s \to \infty} \Psi(e^s \cdot x)  > 0.$$
By continuity, this implies that $(x,z) \in U'_0$; therefore,
$U_0 \times \C \subset U'_0$.
Moreover, by \eqref{Mcut} and \eqref{Psicut},
\begin{equation*}\label{e:cut level set}
\Psi_{\cut}^{-1}(0) =  \{ (x,0) \in M \times \CC \mid \Psi(x) = 0\}/\SS^1. 
\end{equation*}
Hence, the $\CC^\times$-equivariant biholomorphism $M_{\cut} \to U'_0/\CC^\times$ maps
$\Psi_{\cut}\inv(0)$  into 
$U_0 \times_{\CC^\times} \CC \subset U'_0/\CC^\times.$
Since $U_0 \times_{\CC^\times} \CC$ is open by Lemma~\ref{l:monotone2},
this proves Claim (2).

It is straightforward to check that 
the  map from $\Psi\inv(-\infty,0)$ to $\Psi_{\cut}\inv(-\infty,0)$ that
sends $x$ to $\big[x, \sqrt{ - 2 \Psi(x)}\,\big]$ for all $x \in \Psi\inv(-\infty,0)$
is an $\SS^1$-equivariant symplectomorphism that
intertwines the moment maps.
Given a regular $s < 0$, it  restricts to an $\SS^1$-equivariant  diffeomorphism from 
$\Psi\inv(s)$ to $\Psi_{\cut}\inv(s)$, and so induces a diffeomorphism from
$M \red{s} \SS^1$ to $M_{\cut} \red{s} \SS^1$.
Under the identification $TM' \cong TM \times T\CC$, 
it  sends $X \in T_x (\Psi\inv(s))$ 
to $(X,0) \in T_{[x, \sqrt{-2s}]}  (\Psi_{\cut}\inv(s)) .$
Let $\Xi$ be defined by \eqref{e:X};  
since \eqref{cutiso} is an isomorphism, 
Proposition~\ref{p:reduction} implies that
the natural map from
\begin{equation}\label{e:ugly}
\{X'\in T_{(x,\sqrt{-2s})}M'\mid \omega'(\xi_{M'},X')=\omega'(\xi_{M'},J'(X'))=\omega'(\Xi,X')=\omega'(\Xi,J'(X'))=0\}
\end{equation}
to $T_{[x,\sqrt{-2s}]}(M_{\cut}\red{s}\SS^1)$ is an isomorphism of complex and symplectic vector spaces.
Under the identification $TM'\cong TM\times T\CC$, the vector space in \eqref{e:ugly} can
be rewritten as 
$$\{(X,0)\in T_{(x,\sqrt{-2s})} M' \mid \omega(\xi_{M},X)=\omega(\xi_{M},J(X))=0\}.$$ 
Therefore, Claim (3) follows from part (3) of Proposition~\ref{p:reduction}.

Finally,  fix $s \in \R$ and assume that $J$ tames $\omega$ near $\Psi^{-1}(s)$.  
If $[x,z] \in \Psi_{\cut}\inv(s)$, then $x \in \Psi\inv(a)$ by \eqref{Psicut}.
Hence,  $J$ tames $\omega$ near $x$, and
so $J'$ tames $\omega'$ near $(x,z)$. 
Thus $J_{\cut}$ tames $\omega_{\cut}$ near $[x,z]$ by  \eqref{cutiso}.
This proves Claim (4).
\end{proof}

\begin{remark}\label{r:toriccutting}
Let the circle 
$\SS^1 \times \{1\}^{n-1} \subset (\SS^1)^n$ 
act on a symplectic toric manifold
$(M,\omega,\Phi)$  with moment polytope $\Delta$,
as described
in Remark~\ref{r:toricreduction}, satisfying the assumptions of Proposition~\ref{p:cutting}.
In this case, the cut space $M_{\cut}$ a symplectic toric orbifold with
moment polytope 
$$\Delta_{\cut} := \Delta \cap \{ x \in \R^n \mid x_1 \leq 0\};$$
moreover,  the fixed component $\Psi^{-1}_{\cut}(0)$  maps to the new facet  $\Delta \cap \big(\{0\} \times \R^{n-1}\big).$
\end{remark}

\begin{remark}\label{r:reverse action}
Let $(M,J)$ be a complex manifold with a holomorphic $\CC^\times$-action,
a symplectic form $\omega \in \Omega^2(M)^{\SS^1}$ tamed by the action, and a moment map $\Psi \colon M \to \R$. 
Then the {\bf reversed} $\CC^\times$-action on $M$, given by
$(\lambda, x)  \mapsto 
\lambda\inv \cdot x$,
is also holomorphic.
Since $-\xi_M$ is the induced vector field for the associated $\SS^1 \subset \CC^\times$-action,
the symplectic
form $\omega \in \Omega^2(M)^{\SS^1}$ is tamed by the reversed action, and
$\Psi' := - \Psi$ is a moment map for it.
The reduced space $(\Psi')\inv(s)/\SS^1$ is naturally biholomorphically symplectomorphic to
the reduced space $\Psi\inv(-s)/\SS^1$  for all $s \in \R$.
However,  
the Euler classes of the  
bundles $\CC^\times \cdot (\Psi')\inv(s) \to (\Psi')\inv(s)/\SS^1$ 
and  $\CC^\times \cdot \Psi\inv(-s) \to \Psi\inv(-s)/\SS^1$ are additive inverses, as are the weights of the original and reversed action at each fixed point.

Thus, by reversing the action, applying Proposition~\ref{p:cutting},
and reversing the action again,  we see that
Proposition~\ref{p:cutting} still holds with
the following modifications: Replace $(-\infty,0]$ by $[0,\infty)$ in Claim (1);
replace the diagonal action by the antidiagonal action in Claim (2);
and  replace $(-\infty,0)$ by $(0,\infty)$ in Claim (3).
See Remarks~\ref{r:reverse 2} and~\ref{r:reverse 3} for further applications.
\end{remark}

\begin{remark}\label{r:cutting}
Cutting can be used to compactify manifolds.
Let $(M,J)$ be a complex manifold with a holomorphic $\CC^\times$-action,
a symplectic form $\omega \in \Omega^2(M)^{\SS^1}$ tamed by the action, and a 
proper moment map $\Psi \colon M \to \R$. 
If  $a < b$ are regular values, then
by applying Proposition~\ref{p:cutting} twice 
(once modified as in Remark~\ref{r:reverse action}), we get a compact complex manifold $(M',J')$ 
with a holomorphic $\C^\times$-action, a symplectic form $\omega'\in\Omega^2(M')^{\SS^1}$ tamed by the action, and a moment map $\Psi' \colon M'\to\RR$, so that
$\Psi'(M') \subseteq [a,b]$ and  the appropriate analogues of Claims (2)-(4) hold.
\end{remark}

\section{Tame Blow-Ups}\label{s:blowup}


It is well-known that the blow-ups of K\"ahler manifolds admit K\"ahler forms.
In this section, we generalise blow-ups to our setting.
In particular, we first show that the blow-up at one point of a complex manifold 
with
a tamed symplectic form  admits a tamed symplectic form, and then extend this claim to
the blow-up of an isolated $\ZZ_2$-singularity in a complex orbifold.

\begin{proposition}\label{p:blowup2-mfld}
Let $(M,J)$ be a complex manifold with a holomorphic $\CC^\times$-action, a symplectic form $\omega\in\Omega^2(M)^{\SS^1}$ tamed by the action (everywhere) and tamed by $J$ on $W \subseteq M$, and a moment map $\Psi\colon M\to\RR$.  
Let
 $(\widehat{M},\widehat{J})$ be the complex blow-up of $M$ at $p\in M^{\SS^1}\cap W$.
For sufficiently small $t>0$, there exist a symplectic form $\widehat{\omega}\in\Omega^2(\widehat{M})^{\SS^1}$ tamed by the action (everywhere) and tamed by $\widehat{J}$ on $q^{-1}(W)$, 
and a moment map $\widehat{\Psi}\colon \widehat{M}\to\RR$ such that $$[\widehat{\omega}]=q^*[\omega]-t\mathcal{E},$$ 
where $q \colon \widehat{M}\to M$ is the blow-down map  
and $\mathcal{E}$ is the Poincar\'e dual of the exceptional divisor $q^{-1}(p)$.
 Moreover, given a neighbourhood $V$ of $p$, we may assume that $\widehat{\omega}=q^*\omega$ and $\widehat{\Psi}=q^*\Psi$ on $\widehat{M} \smallsetminus q\inv(V)$.
\end{proposition}

\begin{proof}
We may assume that $W$ is open.
By Lemma~\ref{l:bochner}, there exists an $\SS^1$-equivariant biholomorphism from an $\SS^1$-invariant neighbourhood of  $0 \in \CC^n$ 
to a neighbourhood $U\subseteq W\cap V$ of $p$, 
where $\SS^1$ acts on $\CC^n$ with weights $(\alpha_1,...,\alpha_n)$.  
We identify these neighbourhoods, and also identify $q^{-1}(U)$ with a neighbourhood of the exceptional divisor $E$ in 
$$\widehat{\CC^n}:=(\CC^n\smallsetminus\{0\})\times_{\CC^\times}\CC,$$ 
where $\CC^\times$ acts on  
$(\CC^n \smallsetminus \{0\}) \times \C$ by
$$\lambda\cdot(z_1,...,z_n;u)=(\lambda z_1,...,\lambda z_n;\lambda^{-1}u).$$  
In these coordinates, the blow-down map $q\colon \widehat{\CC^n}\to\CC^n$ sends $[z_1,...,z_n;u]$ to $(uz_1,...,uz_n)$. 
Define $\pi\colon \widehat{\CC^n}\to\CC\PP^{n-1}$ by  $\pi\big([z_1,...,z_n;u]\big) =[z_1,...,z_n]$. 

Since the restriction $q\colon \widehat{M}\smallsetminus E\to M\smallsetminus\{p\}$ is $\CC^\times$-equivariant and biholomorphic, 
the closed form $q^*\omega\in\Omega^2(\widehat{M})^{\SS^1}$ is symplectic on $\widehat{M}\smallsetminus E$, tamed by the action on $\widehat{M} \smallsetminus E$, and
 tamed on $q^{-1}(W)\smallsetminus E$.
 Since $q$ is holomorphic and $\ker(\pi_*|_m)\cap\ker(q_*|_m)=\{0\}$, 
$$q^*\omega(X,\widehat{J}(X))\geq 0$$ 
for all $m\in E$ and $X\in T_m\widehat{M}$, with equality impossible if $\pi_*X=0$ and $X \neq 0$. 
Finally, since $q \colon \widehat{M}\to M$ is equivariant, 
$\xi_{\widehat{M}} \hook q^*\omega = - d q^* \Psi$.

Let $\rho \colon \RR \to \RR$ be a smooth function which is $1$ on a neighbourhood of $0$ and such that $z\mapsto \rho(|z|^2)$ has compact support in $U$. 
 Define $f\colon \RR^+\to\RR$ by $f(t)=\frac{1}{2 \pi} \rho(t)\ln t$ (\emph{cf.} \cite[Section 5]{GS89}). 
 As shown in \cite[Section 4]{GS89}, on the complement of $E$, 
$$\pi^*(\Omega)=q^*\Big(\frac{i}{2\pi}\del\delbar\ln\big(|z|^2\big)\Big),$$ 
where $\Omega$ is the Fubini-Study form on $\CC\PP^{n-1}$.  
Therefore, there exists a closed real form $\eta\in\Omega^{1,1}(\widehat{M})^{\SS^1}$ 
with support in $U$  equal to $q^*(i \del \delbar(f(|z|^2)))$ on $q^{-1}(U)\smallsetminus E$, and equal to $ \pi^*(\Omega)$ near $E$.  
Define $\Phi:\widehat{M}\to\RR$ with support in $U$  by 
$$\Phi(z)=q^*\bigg(\sum_{j=1}^n \alpha_j|z_j|^2f'(|z|^2)\bigg)$$ on $q^{-1}(U)\smallsetminus E$ and by 
\begin{equation}\label{e:coord-Phi}
 \pi^*\Bigg(\frac{\sum_{j=1}^n \alpha_j|z_j|^2}{2 \pi |z|^2}\Bigg)
\end{equation}
near $E$. 
Then $\xi_{\widehat{M}}\hook\eta=-d\Phi$
by a straightforward calculation on $U \smallsetminus \{0\}$.
The restriction of $[\eta]\in H^2(\widehat{M})$ to $E$ is the positive
generator of $H^2(E;\Z) \cong \Z$,
which is the negative of the Euler class of the normal bundle
to $E$ in $\widehat{M}$.
Moreover, since $\eta$ is supported in a tubular neighbourhood of $E$,
 the restriction of $[\eta]$ to $\widehat{M} \smallsetminus E$ vanishes.  
Hence,  $[\eta] = - \mathcal{E}$.

Therefore, for all $t\in\RR$, $$[q^*\omega+t\eta]=q^*[\omega]-t\mathcal{E} \text{ and }\xi_{\widehat{M}}\hook(q^*\omega+t\eta)=-d(q^*\Psi+t\Phi).$$
It remains to show that $q^*\omega+t\eta\in\Omega^2(\widehat{M})^{\SS^1}$ is symplectic, is tamed by the action (everywhere), and is
tamed on $q^{-1}(W)$, for all sufficiently small $t>0$.  We will do this by looking at three regions.
\begin{enumerate}
\item On a neighbourhood of the exceptional divisor, $\eta=\pi^*(\Omega)$.  Since $\pi$ is holomorphic and $\Omega$ is K\"ahler, this implies that $\eta(X,\widehat{J}(X))\geq 0$ for all $X\in T_m\widehat{M}$, with equality exactly if $\pi_*(X)=0$. 
 Moreover, by the second paragraph of this proof, $q^*\omega(X,\widehat{J}(X))\geq 0$ with equality impossible if $\pi_*X=0$ and $X \neq 0$.  
Therefore, for all $t > 0$,
$$\big(q^*\omega+t\eta)(X,\widehat{J}(X)\big)>0$$ 
for all non-zero $X\in T_m\widehat{M}$, that is,
$q^*\omega + t \eta$ is tamed.

\item On the complement of $q^{-1}(\supp(\rho))\subseteq\widehat{M}$, the form $\eta$ vanishes; 
hence, by the second paragraph, $q^*\omega+t\eta$ is symplectic, is tamed by the action (everywhere), and is tamed on $q^{-1}(W)$, for all $t$.

\item The complement of the open sets considered in (1) and (2) is compact.  Since $q^*\omega$ is tamed on this set, $q^*\omega+t\eta$ is also tamed for all sufficiently small $t$.
\end{enumerate}
\end{proof}

\begin{remark}
In Proposition~\ref{p:blowup2-mfld} (or Proposition~\ref{p:blowup2} below), if $\Psi$ is proper, then we may
choose $\widehat{\omega}$ and $\widehat{\Psi}$ so that $\widehat{\Psi}$
is proper.  To see this, let $V$ be a neighbourhood with compact closure.
\end{remark}

The exact same argument shows that Proposition~\ref{p:blowup2-mfld} still holds
if $(M,J)$ is a complex {\em orbifold},
as long as we blow up at a smooth point $p \in M$. 
In order to prove Proposition~\ref{p:prop}, we need to 
extend that  argument  to the
blow-up of a complex orbifold
at an isolated $\Z_2$-singularity.
First, we recall the definition of  blow-up 
in this case.

\begin{definition}\label{d:blow-up}
Let $\ZZ_2$ act diagonally on $\CC^n$.
The \textbf{blow-up} of $\CC^n/\ZZ_2$ at $[0]$ is 
\begin{equation*}\label{eq:blow-up}
\widehat{\CC^n/\ZZ_2}:=(\CC^n\smallsetminus\{0\})\times_{\CC^\times}\CC, 
\end{equation*}
where $\CC^\times$ acts on $(\CC^n\smallsetminus\{0\})\times\CC$ by $$\lambda\cdot(z_1,...,z_n;u)=(\lambda z_1,...,\lambda z_n;\lambda^{-2}u).$$ 
The \textbf{blow-down map} $q\colon \widehat{\CC^n/\ZZ_2}\to\CC^n/\ZZ_2$ is given
by $$q([z_1,...,z_n;u]):=[\sqrt{u}z_1,...,\sqrt{u}z_n].$$  
\end{definition}
Unfortunately, although the map $q$ is continuous and the pullback $q^*f$ is holomorphic for every holomorphic function $f:\CC^n/\ZZ_2\to\CC$, the blow-down map $q$ is not smooth. 
 For example, $[w]\mapsto|w|^2$ is a smooth function on $\CC^n/\ZZ_2$, but its pullback $[z;u]\mapsto |u||z|^2$ is not a smooth function on $\widehat{\CC^n/\ZZ_2}$. 
 (Note that, when $n=1$, $\widehat{\CC/\ZZ_2}\cong\CC$ and $q(u)=[\sqrt{u}\,]$.)
  However, the exceptional divisor $E:=q^{-1}\left([0]\right)$ is biholomorphic to $\CC\PP^{n-1}$, and  
$q$ restricts to a biholomorphism from $\widehat{\CC^n/\ZZ_2}\smallsetminus E$ to $\CC^n/\ZZ_2 \smallsetminus\{[0]\}$. Thus,
there is a well-defined blow-up of a complex orbifold at any isolated $\ZZ_2$-singularity.

Since the blow-down map is not smooth,
we  need to modify the symplectic form locally before pulling it back
to the blow-up. We will do this in two stages, using
the following criterion for K\"ahler forms,
which we   adapted from \cite[Lemma 5.3]{GS89}.

\begin{lemma}\label{l:pointwise kaehler}
Given $n>1$ and a smooth function $f\colon \RR\to\RR$, 
the form $\omega=\frac{i}{2}\del\delbar f\big(|z|^2\big)\in\Omega^{1,1}(\CC^n)$ is K\"ahler at $z_0\in\CC^n$ exactly if 
$$f'\big(|z_0|^2\big)>0 \quad \mbox{and} \quad f'\big(|z_0|^2\big)+|z_0|^2f''\big(|z_0|^2\big)>0.$$
\end{lemma}

\begin{proof}
Define $g\colon \CC^n\to\RR$ by $g(z)=f\big(|z|^2\big)$. 
The form $\omega=\frac{i}{2}\del\delbar g$ is K\"ahler exactly if the eigenvalues of the Hermitian matrix
\begin{equation*}\label{e:hermitian}
\left[\frac{\del^2 g}{\del z_j \, \del\overline{z}_k}\right] = \Big[ f'\big(|z|^2\big)\delta_{jk}+f''\big(|z|^2\big)\overline{z}_jz_k \Big]
\end{equation*}
are all positive.  Since the matrix $[\overline{z}_jz_k]$ is of rank 1 and has eigenvalues 0 and $|z|^2$, 
the matrix above has eigenvalues $f'\big(|z|^2\big)$ and $f'\big(|z|^2\big)+|z|^2f''\big(|z|^2\big).$
\end{proof}

\begin{remark}\label{r:pointwise kaehler}
Similarly, given a smooth function $f\colon \RR\to\RR$, 
the form $\omega=\frac{i}{2}\del\delbar f\big(|z|^2\big)\in\Omega^{1,1}(\CC)$ is K\"ahler at $z_0\in\CC$ exactly if 
$f'\big(|z_0|^2\big)+|z_0|^2f''\big(|z_0|^2\big)>0.$
\end{remark}

The next lemma translates \cite[Theorem 1.10]{reyer} to our setting.

\begin{lemma}\label{l:modifyforms1b}
Let $G$ be a closed subgroup of $\U(n)$.  Let $U$ be a $G$-invariant neighbourhood of $0 \in \CC^n$, 
and let $\omega\in\Omega^2(U)^G$ be a tamed symplectic form. 
 Then there exists $\nu\in\Omega^1(U)^G$ with compact support such that $\widetilde{\omega}:=\omega-d\nu$ is a tamed symplectic form
and is constant in a neighbourhood of $0$.  
Moreover, if $\omega\in\Omega^{1,1}(U)^G$ is  K\"ahler, 
then we may also choose $\nu$ so that $\widetilde{\omega}$
is K\"ahler.
\end{lemma}

\begin{proof}
We may assume that $U$ is a ball centred at $0$. 
Pick a $G$-invariant smooth function $\rho \colon \CC^{n} \to [0,1]$ so that $\rho$ has compact support $K$ in $U$ and is $1$ on a neighbourhood of $0$.
Given $\lambda>0$, define $\rho_\lambda \colon \CC^{n}\to [0,1]$  by $\rho_\lambda(x) = \rho(\lambda x)$; 
the support of $\rho_\lambda$ is $\frac{1}{\lambda} K$.

Identify $\CC^n$ with $\RR^{2n}$.  
Let $\omega'\in\Omega^2(U)^G$ be the unique constant form satisfying $\omega'|_0=\omega|_0$.  By the Poincar\'e Lemma (and averaging), there exists $\mu=\sum_if_idx_i \in \Omega^1(U)^G$ such that $d \mu = \omega-\omega'.$  
 Since $d\mu|_0=0$, we may assume that $f_i(0)=\frac{\partial f_i}{\partial x_j}(0)=0$ for all $i,j$.   
Therefore, by the Mean Value Theorem, there exists $C>0$ so that 
$\frac{\partial f_i}{\partial x_j}(x) \leq C|x|$ 
and 
$f_i(x) \leq C|x|^2$ 
for all $x\in K$ and all $i,j$. 
 Since tameness is an open condition, a straightforward calculation in coordinates
shows that $\omega - d(\rho_\lambda\, \mu)$ is tamed for sufficiently large $\lambda$. 
Let $\nu = \rho_\lambda\,\mu\in\Omega^1(U)^G$.

If $\omega\in\Omega^{1,1}(U)$, then by the Poincar\'e Lemma for $d$ and $\del$,  
there exists a smooth $G$-invariant potential
function $h\colon U\to\RR$ such that $\omega - \omega' = \frac{i}{2} \del \delbar ( \rho_\lambda h).$  
Since $\omega|_0=\omega'|_0$, we may assume that $h(0) = 0$, $\partial h|_0 = 0$, and $\overline{\del} h|_0 = 0$.  
By a calculation similar to the one above, if we
let $\nu := \frac{i}{4}\big(\delbar(\rho_\lambda h)-\del(\rho_\lambda h)\big)
\in \Omega^1(U)^G$, 
then $\widetilde{\omega} := \omega - d \nu$ is
K\"ahler for sufficiently large $\lambda$.
\end{proof}

\begin{lemma}\label{l:modifyforms2}
Let $G$ be a closed subgroup of $\U(n)$. 
Let $U$ be a $G$-invariant neighbourhood of $0 \in \CC^n$, 
and let $\omega\in\Omega^2({U})^G$ be a tamed symplectic form.
Given a smooth function $f \colon \R \to \R$ so that the function
$z \mapsto f(|z|^2)$ is strictly plurisubharmonic on
$\CC^n \smallsetminus \{0\}$,
there exists $\nu\in\Omega^1({U})^G$ with compact support such that $\widetilde{\omega}:=\omega-d\nu$ satisfies the following:
\begin{enumerate}
\item $\widetilde{\omega}$ is a tamed symplectic form
on ${U} \smallsetminus \{0\}$; and
\item  
$\widetilde{\omega}^{1,1}=
\frac{i}{2} \del \delbar f \big( |z|^2 \big)$, and 
$\widetilde{\omega}^{2,0}$ and $\widetilde{\omega}^{0,2}$ are constant,
on a neighbourhood of $0 \in U$.
\end{enumerate}
Moreover, if $\omega\in\Omega^{1,1}({U})^G$ is K\"ahler, then we may 
also choose $\nu$ so that $\widetilde{\omega}$ is K\"ahler on $U \smallsetminus \{0\}$.
\end{lemma}

\begin{proof}
By Lemma~\ref{l:modifyforms1b}, we may assume that $\omega$ is constant.  Write $\omega=\omega^{1,1}+\omega^{2,0}+\omega^{0,2}$, where $\omega^{j,k}\in\Omega^{j,k}(\CC^n)$  for all $j,k$.
By a linear change of variables, we may further assume that $\omega^{1,1}=\frac{i}{2}\sum_jdz_j\wedge d\overline{z}_j =  \frac{i}{2} \del \delbar |z|^2$ on ${U}$.
 Finally, ${U}$ contains a closed ball $B_r$ of radius $r > 0$ centred at $0$.

Assume first that $n > 1$.
By assumption, 
the form $\frac{i}{2} \del \delbar {f} \big( |z|^2 \big)$ is K\"ahler on  $\CC^n \smallsetminus \{0\}$.
By Lemma~\ref{l:pointwise kaehler}, this implies that $f'(t)$ and $t f''(t) + f'(t)$
are positive for all $t > 0$.
Hence, there exists a smooth function $h\colon \RR\to\RR$ such that $h(t)= t f'(t)$ in a neighbourhood of $0$, $h(t)=t$ for all $t>r$, and $h(t)$ and $h'(t)$ are positive for all $t>0$. 
Let $\widetilde{f}\colon \RR\to\RR$  be the smooth function with
$\widetilde{f}'(t)=h(t)/t$ and $\widetilde{f}(0) = f(0)$.  Then $\widetilde{f}(t) = f(t)$ near $0$, $\widetilde{f}'(t)=1$ for all $t>r$, and $\widetilde{f}'(t)$ and $t\widetilde{f}''(t)+\widetilde{f}'(t)$ are positive for all $t>0$. 
Hence, Lemma~\ref{l:pointwise kaehler} implies that the form
$\frac{i}{2} \del \delbar \widetilde{f} \big( |z|^2 \big)$ is K\"ahler on  $\CC^n \smallsetminus \{0\}$.
By Remark~\ref{r:pointwise kaehler}, a similar argument applies if $n = 1$.

Define $g\colon \CC^n\to\RR$ by $g(z):=|z|^2-\widetilde{f}(|z|^2)$, 
and let $\nu:=\frac{i}{4}(\delbar g-\del g)\in\Omega^1(\CC^n)^{\U(n)}.$
Define $\widetilde{\omega} := \omega - d \nu \in \Omega^2(U)^G$.
Since $\widetilde{f}'(t)=1$ for all $t>r$, the support of $\nu$ is contained in $B_r\subset U$.  
Since $d \nu = \frac{i}{2} \del \delbar g \in \Omega^{1,1}(\CC^n)^{\U(n)}$,
we have $$\widetilde{\omega}^{1,1}=
\frac{i}{2}\del\delbar \widetilde{f}\big(|z|^2\big), \quad
\widetilde{\omega}^{0,2} = \omega^{0,2}, \quad  \mbox{and} \quad \widetilde{\omega}^{2,0} = \omega^{2,0}.$$   
This shows that $\widetilde{\omega}^{1,1}$ is K\"ahler on $U \smallsetminus \{0\}$,
and proves Claim (2).
If  $J$ is the standard complex structure on $\CC^n$,  
then $$\omega^{2,0}(X,JX)=\omega^{0,2}(X,JX)=0$$ for all vectors $X\in TU$.
This proves Claim (1).
The last claim follows immediately.

\end{proof}

We are now ready to extend Proposition~\ref{p:blowup2-mfld} to 
the blow-up of a complex orbifold at an isolated $\Z_2$-singularity.

\begin{proposition}\label{p:blowup2}
Let $(M,J)$ be a complex orbifold with a holomorphic $\CC^\times$-action, a symplectic form $\omega\in\Omega^2(M)^{\SS^1}$ tamed by the action (everywhere) and tamed by $J$ on $W\subseteq M$, and a moment map $\Psi\colon M\to\RR$.  
Let $\big(\widehat{M},\widehat{J}\big)$ be the complex blow-up of $M$ at an isolated $\ZZ_2$-singularity $p\in M^{\SS^1}\cap W$.
 For sufficiently small $t>0$, there exist a symplectic form $\widehat{\omega}\in\Omega^2(\widehat{M})^{\SS^1}$ tamed by the action (everywhere) and tamed by $\widehat{J}$ on $q^{-1}(W)$, 
and a moment map $\widehat{\Psi}\colon \widehat{M}\to\RR$ such that $$[\widehat{\omega}]=q^*[\omega]-\frac{t}{2}\mathcal{E},$$ 
where $q:\widehat{M}\to M$  the blow-down map
and $\mathcal{E}$ is the Poincar\'e dual of the exceptional divisor $q^{-1}(p)$.  
Moreover, given a neighbourhood $V$ of $p$, we may assume that $\widehat{\omega}=q^*\omega$ and $\widehat{\Psi}=q^*\Psi$ on $\widehat{M} \smallsetminus q\inv(V)$.
\end{proposition}

\begin{proof}
We may assume that $W$ is open.
By Lemma~\ref{l:bochner}, there exists an $\SS^1$-equivariant biholomorphism from an $\SS^1$-invariant neighbourhood of $[0] \in \CC^n/\ZZ_2$ to a neighbourhood $U\subseteq W\cap V$ of $p$, 
where $\Z_2$ acts diagonally on $\CC^n$ and $\SS^1$ acts  with weights 
$(\alpha_1,...,\alpha_n)$.
We identify these neighbourhoods, and also identify $q^{-1}(U)$ with a neighbourhood of the exceptional divisor $E$ in $\widehat{\CC^n/\ZZ_2}:=(\CC^n\smallsetminus\{0\})\times_{\CC^\times}\CC$. (See Definition~\ref{d:blow-up}.)
Define $\pi\colon \widehat{\CC^n/\ZZ_2}\to\CC\PP^{n-1}$ by 
$\pi([z_1,...,z_n;u]) =
[z_1,...,z_n]$.

The function $z \mapsto |z|^4$ is strictly plurisubharmonic on $\CC^n \smallsetminus \{0\}$.
Hence, by Lemma~\ref{l:modifyforms2} there exists a closed $\SS^1$-invariant form $\widetilde{\omega}\in[\omega]\in H^2(M)$ so that $\widetilde{\omega}=\omega$ on $M \smallsetminus U$; 
moreover $\widetilde{\omega}$ satisfies the following:
\begin{enumerate}
\item $\widetilde{\omega}$ is tamed on $U \smallsetminus \{p\}$; and
\item $\widetilde{\omega}^{1,1}=\frac{i}{2}\del\delbar(|z|^4)$,
and $\widetilde{\omega}^{2,0}$ and $\widetilde{\omega}^{0,2}$ are constant, on a neighbourhood of $p$.
\end{enumerate}
Let $\widetilde{\Psi}:M\to\RR$ be the smooth function
satisfying $\xi_M\hook\widetilde{\omega}=-d\widetilde{\Psi}$ so that $\widetilde{\Psi}=\Psi$ on $M \smallsetminus U$.
  Since the restriction $q\colon \widehat{M}\smallsetminus E\to M\smallsetminus\{p\}$ is $\CC^\times$-equivariant and biholomorphic, 
the form $q^*\widetilde{\omega}\in\Omega^2(\widehat{M}\smallsetminus E)^{\SS^1}$ is symplectic, tamed on $q^{-1}(W)\smallsetminus E$, and satisfies $q^*\widetilde{\omega}(\xi_{\widehat{M}},\widehat{J}(\xi_{\widehat{M}}))>0$ on $\widehat{M}\smallsetminus\big(\widehat{M}^{\SS^1}\cup E\big)$.
  A straightforward calculation in local coordinates shows that  (2) implies 
that there exists a unique closed form on $\widehat{M}$ that restricts to $q^*\widetilde{\omega}$ on $\widehat{M}\smallsetminus E$; 
by a slight abuse of notation, we will call this form $q^*\widetilde{\omega}\in\Omega^2(\widehat{M})^{\SS^1}$.  
Moreover, $$(q^*\widetilde{\omega})^{1,1}=\frac{i}{2}\del\delbar \big(|u|^2|z|^4 \big)$$ on a neighbourhood of $E$.
 Hence, another straightforward calculation  implies that for all $m\in E$ and $X\in T_m\widehat{M}$, $$q^*\widetilde{\omega}(X,\widehat{J}(X))\geq 0,$$ with equality impossible if $\pi_*X=0$ and $X \neq 0$. 
 Finally, since $q:\widehat{M}\smallsetminus E\to M\smallsetminus\{p\}$ is smooth and equivariant, $\xi_{\widehat{M}}\hook q^*\widetilde{\omega}=-dq^*\widetilde{\Psi}$ on $\widehat{M}\smallsetminus E$. 
 Since $q^*\widetilde{\Psi}$ is continuous and $\widehat{M}\smallsetminus E$ is dense in $\widehat{M}$, this implies that $\xi_{\widehat{M}}\hook q^*\widetilde{\omega}=-dq^*\widetilde{\Psi}$ on $\widehat{M}$.

The remainder of the proof is nearly identical to the proof of Proposition~\ref{p:blowup2-mfld}.  
The main distinction is that here, the Euler class to the normal bundle of $E$  in $\widehat{M}$ is twice the (negative) generator of $H^2(E;\ZZ)$. 
 Thus, if we construct $\eta$ as in the proof of Proposition~\ref{p:blowup2-mfld}, 
then $[\eta] =- \frac{1}{2} \mathcal{E}$.  Hence, 
$[\widetilde{\omega}+t\eta] =q^*[\omega]-\frac{t}{2}\mathcal{E}$.
\end{proof}

\begin{remark}\label{r:toricblowup}
Let $(M,\omega,\Phi)$  be a symplectic toric orbifold with moment
polytope $\Delta$, as described
in Remark~\ref{r:toricreduction}.
Let $\eta_1,\dots,\eta_n \in \Z^n$ be the primitive outward normals to the 
facets  that intersect at a vertex $v \in \Delta$, and assume the natural number
associated to each of these $n$ facets is $1$.
The preimage $\Phi\inv(v)$  is smooth exactly if
$\eta_1,\dots,\eta_n$ form a basis for $\Z^n$.
In contrast, it is an isolated $\Z_2$-singularity
 exactly if 
$\frac{1}{2}(\eta_1 + \dots + \eta_n) \in \Z^n$ and
$\eta_1,\dots,\eta_n$ generate a sublattice of $\Z^n$ of index $2$.
In the former case,
for sufficiently small  $t > 0$,
the moment polytope of the blow-up $(\widehat{M},\widehat{\omega})$
of $M$ at $\Phi^{-1}(v)$ 
with $[\widehat{\omega}] = q^*[\omega] - t \mathcal{E}$
is $$\widehat{\Delta} = \Delta \cap \left\{ x \in \R^n \; \left| \; \sum_{i=1}^n \langle \eta_i, x \rangle \leq \sum_{i = 1}^n 
\langle \eta_i, v \rangle - \frac{t}{2 \pi} \right\} \right. .$$
The same claim holds in the latter case
except now $[\widehat{\omega}] = q^*[\omega] - \frac{t}{2}\mathcal{E}$.
\end{remark}

\section{Adding Fixed Points to Tame Actions}\label{s:application}


We are now ready to build the specific machinery that the first author
needs to construct a non-Hamiltonian symplectic circle action with isolated
fixed points on a closed, connected symplectic manifold in \cite{T}.

\begin{proposition}\label{p:prop}
Let $(M,J)$ be a complex manifold with a holomorphic $\CC^\times$-action, a symplectic form $\omega\in\Omega^2(M)^{\SS^1}$tamed by the action (everywhere) and tamed by $J$ near $\Psi^{-1}(0)$, and a proper moment map $\Psi\colon M\to\RR$.  
Assume that the $\SS^1$-action on $\Psi^{-1}(0)$ is free except for $k$ orbits with stabilisers $\ZZ_2$. 
Then for  sufficiently small $\eps>0$ there exist a complex manifold $\big(\widetilde{M},\widetilde{J}\big)$ with a holomorphic $\CC^\times$-action, a symplectic form $\widetilde{\omega}\in\Omega^2\big(\widetilde{M}\big)^{\SS^1}$ 
tamed by the action, 
and a proper moment map $\widetilde{\Psi}\colon \widetilde{M}\to\RR$ so that the following hold:
\begin{enumerate}
\item $\widetilde{\Psi}^{-1}(-\eps,0]$ contains exactly $k$ 
fixed points; each lies in $\widetilde{\Psi}\inv(0)$ and
has weights $\{-2,1,\dots,1\}$.
\item There is an $\SS^1$-equivariant symplectomorphism from $\widetilde{\Psi}^{-1}(-\infty,-\eps/2)$ to $\Psi^{-1}(-\infty,-\eps/2)$ 
that induces a biholomorphism from $\widetilde{M}\red{s}\SS^1$ to $M\red{s}\SS^1$ for all regular $s \in (-\infty, - \eps/ 2)$.
\item $\widetilde{\omega}$ is  tamed on $\widetilde{\Psi}\inv(-\eps,\eps)$.
\end{enumerate}
\end{proposition}

\begin{proof}
Let $(M_{\cut},J_{\cut})$ be the complex orbifold with holomorphic $\CC^\times$-action, symplectic form $\omega_{\cut}\in\Omega^2(M_{\cut})^{\SS^1}$ 
tamed by the action, and proper moment map $\Psi_{\cut}\colon M_{\cut}\to\RR$ described in Proposition~\ref{p:cutting}.  
Then $\Psi_{\cut}(M_{\cut})\subseteq(-\infty,0]$.
Additionally,  a neighbourhood of the fixed component $F := \Psi_{\cut}\inv(0)$ in
$M_{\cut}$ is $\C^\times$-equivariantly biholomorphic to the complex line bundle 
$U_0\times_{\CC^\times}\CC \stackrel{\pi}{\to} U_0/\C^\times$, 
where $U_0 := \CC^\times \cdot \Psi^{-1}(0)$, $\C^\times$ acts
diagonally on $U_0 \times \C$, and $\C^\times$ acts on $U_0 \times_{\CC^\times} \CC$ by 
$\lambda \cdot[ y,z] = [\lambda \cdot y,z] = [y, \lambda^{-1} z]$ for all $ \lambda \in \CC^\times$ and $[y,z] \in U_0 \times_{\CC^\times} \CC$;
we identify these manifolds.
In particular, 
$F$ is diffeomorphic to $M\red{0}\SS^1 = U_0/\CC^\times$. 
Moreover, there exists an $\SS^1$-equivariant symplectomorphism from $\Psi^{-1}(-\infty,0)$ to $\Psi_{\cut}^{-1}(-\infty,0)$ that intertwines the moment maps and induces a biholomorphism between the reduced spaces at all regular $s<0$.  
Hence, the orbifold $M_{\cut}$ is smooth except at isolated $\ZZ_2$-singularities $p_1,...,p_k\in F$. 
Finally, $J_{\cut}$ tames $\omega_{\cut}$ on a  neighbourhood 
$W \subset U_0 \times_{\CC^\times} \CC$ of $F$.

There exists $\eps > 0$ so that $\Psi_{\cut}\inv(-\eps,0] \subseteq W$
and $\Psi_{\cut}\inv(-\eps,0)$ contains no fixed points.
Since $\Psi_{\cut}(M_{\cut})\subseteq(-\infty,0]$, 
we may also assume that $\Psi_{\cut}^{-1}(-\infty,-\eps/2)\neq\emptyset$. 

Let $Q \colon U_0 \times_{\CC^\times} \CC \to \R$ be the quadratic form associated to a Hermitian metric on 
the line bundle $U_0\times_{\CC^\times}\CC$,
and let $\rho \colon \R \to \R$ be a smooth function  which is
$1$ on a neighbourhood of $0$ and such that the support of $x \mapsto \rho(Q(x) )$
is contained in $\Psi_{\cut}^{-1}(-\eps/2,0]$.
Define $h \colon U_0 \times_{\CC^\times} \left(\CC \smallsetminus \{0\} \right) \to \R$ by $h(x):=\rho(Q(x)) \ln Q(x).$ 
Since the closed real form $i \del \delbar \ln Q \in \Omega^{1,1}(U_0 \times_{\CC^\times} \left(\CC \smallsetminus \{0\} \right) )$ is basic,
there exists a closed real form $\Omega \in \Omega^{1,1} (U_0/\CC^\times)$ 
such that $\pi^*(\Omega) = i \del \delbar \ln Q$.
Therefore, there exists a closed real form
$\beta \in \Omega^{1,1}(M_{\cut})$
with $\supp (\beta) \subset \Psi_{\cut}\inv(-\eps/2,0]$ 
that is equal to $i \del \delbar h$ on 
 $U_0 \times_{\CC^\times} (\CC \smallsetminus \{0\})$, and equal
to $\pi^*(\Omega)$ near the zero section $F = U_0/\CC^\times$.
Similarly, there exists a smooth function
$\chi \colon M_{\cut}  \to \R$   
with $\supp (\chi) \subset \Psi_{\cut}\inv(-\eps/2,0]$ 
that is equal to
$i\xi_{M_{\cut}} \hook \left(\delbar h - \del h\right)/2$
on 
$U_0 \times_{\CC^\times} (\CC \smallsetminus \{0\})$, and equal
to $1$  near the zero section $F = U_0/\CC^\times$.
Then $\xi_{M_{\cut}} \hook \beta  = - d \chi$ by a straightforward calculation.
Moreover, by Lemma~\ref{l:modifyforms} below, we may assume that $\beta$ 
vanishes (and so $\chi$ is $1$) on pairwise disjoint neighbourhoods $V_j \subset \Psi\inv_{\cut}(-\eps/2,0]$ of  $p_j$ for all $j$.
 Since the support of $\beta$ is compact, there exists $\delta>0$ such that 
for all $t' \in (-\delta, \delta)$ the form
$\omega_{\cut}+t'\beta$ is symplectic, tamed by the action (everywhere), and tamed on $W$. 

Let $\big(\widetilde{M},\widetilde{J}\big)$ be the complex blow-up of 
$M_{\cut}$ at $p_1,...,p_k$, 
let $q\colon \widetilde{M}\to M_{\cut}$ be the blow-down map, 
and let $\mathcal{E}_j$ be the Poincar\'e dual to the exceptional divisor 
$E_j:=q^{-1}(p_j)$ for all  $j$.  
By Proposition~\ref{p:blowup2}, for each $j$ and for 
sufficiently small $t_j>0$ there exist a symplectic form 
$\widetilde{\omega}_j\in\Omega^2\big(\widetilde{M}\smallsetminus\bigcup_{l\neq j}E_l\big)^{\SS^1}$ 
that is tamed by the action (everywhere) and tamed on $q^{-1}(W)$,
 and a moment map $\widetilde{\Psi}_j:\widetilde{M}\smallsetminus\bigcup_{l\neq j}E_l\to\RR$ such that 
\begin{equation}\label{omegaj}
[\widetilde{\omega}_j]=q^*[\omega_{\cut}]-\frac{t_j}{2}\mathcal{E}_j\in H^2\Big(\widetilde{M}\smallsetminus\bigcup_{l\neq j}E_l\Big). 
\end{equation}
Moreover, $\widetilde{\omega}_j=q^*\omega_{\cut} \text{ and } \widetilde{\Psi}_j=q^*\Psi_{\cut}$ on the complement of a closed subset of $q\inv(V_j)$.
Given $t'\in(-\delta,\delta)$, we can glue these forms together to construct a symplectic 
form $\widetilde{\omega}\in\Omega^2\big(\widetilde{M}\big)^{\SS^1}$ and moment map $\widetilde{\Psi}:\widetilde{M}\to\RR$ 
that are equal to $q^*(\omega_{\cut}+t'\beta)$ and $q^*(\Psi_{\cut}+t'\chi)$, 
respectively, on $\widetilde{M}\smallsetminus\bigcup_jq^{-1}(V_j)$ 
and equal to $\widetilde{\omega}_j$ and 
$\widetilde{\Psi}_j+t'$ on $q^{-1}(V_j)$ for each $j$.
By construction,  
$\widetilde{\omega}$ is tamed by the action (everywhere) and tamed
by $J$ on $W$.
Moreover, $\widetilde{\omega} = q^*\omega_{\cut}$  and
$\widetilde{\Psi} = q^*\Psi_{\cut}$ on $\widetilde{M} \smallsetminus (q^*\Psi_{\cut})\inv(-\eps/2,0]$.
Since $q^*$ and $\Psi_{\cut}$ are proper, and since $[-\eps/2,0] \supset (-\eps/2,0]$ is compact,
this implies that the moment map  $\widetilde{\Psi}$ is proper. 

By Lemma~\ref{l:bochner}, for each $j \in \{1,\dots,k\}$ 
there exists an $\SS^1$-equivariant biholomorphism from
an $\SS^1$-invariant neighbourhood of $[0] \in \CC^n/\Z_2$ to a neighbourhood
of $p_j \in M$, where $\Z_2$ acts diagonally on $\CC^n$ and $\SS^1$ acts
with weights $(-1,0,\ldots,0)$.
In the coordinates described in Definition~\ref{d:blow-up},
the corresponding $\SS^1$-action on the blow-up $\widehat{\CC^n/\Z_2}$ is given by
$$ \lambda \cdot [z_1,\dots,z_n;u] = [\lambda\inv z_1, z_2, \dots, z_n;u]
= [z_1,\lambda z_2, \dots, \lambda z_n; \lambda^{-2} u]$$
for all $\lambda \in \SS^1$ and $[z_1,\dots,z_n;u] \in \widehat{\CC^n/\Z_2}$.
In particular, $[z_1,\dots,z_n;u] \in \widehat{\CC^n/\Z_2}$ is fixed exactly if either
$z_2 = \dots = z_n = u = 0$ or $z_1 = 0$.
Therefore, $q\inv(F) \cap M^{\SS^1}$ consists of $k$ isolated fixed points 
$\{\widetilde{p_j} \in E_j\}_{j=1}^k$,
each with  weights $\{-2,1,\dots,1\}$, and a component $\widetilde{F}$
that is the blow-up of $F$ at $p_1,\dots,p_k$.
By \eqref{omegaj},  the restriction of $[\widetilde{\omega}]$ to $E_j$ is $t_j$ times the positive generator $[\Omega] \in H^2(\CC\PP^{n-1};\ZZ)$; hence,
$$\widetilde{\Psi}\big(\widetilde{F}\big)=\widetilde{\Psi}\big(E_j\cap\widetilde{F}\big)=\widetilde{\Psi}(\widetilde{p}_j)+\frac{t_j}{2 \pi}.$$
(Compare with \eqref{e:coord-Phi}.)
Additionally,
$\widetilde{\Psi}(\widetilde{F})=\Psi_{\cut}(F)+t'\chi(F)=t'$. 
We can  fix $t'\in(0,\delta)$ so that  $t_j = 2 \pi t'$ is sufficiently small 
for all $j$.
Then $\widetilde{\Psi}(\widetilde{p}_j)=0$ for all  $j$ 
and $\widetilde{\Psi}(\widetilde{F})>0$. 
Moreover, if $F' \subset \widetilde{M}$ is any other fixed component,
then $q^*\Psi_{\cut}(F') \leq -\eps$ and so 
 $\widetilde{\Psi}(F')=q^*\Psi_{\cut}(F') \leq -\eps$.
 This proves Claim (1).  

Since
$\widetilde{\omega} = q^* \omega_{\cut}$ and
$\widetilde{\Psi}=q^*\Psi_{\cut}$ on  
$\widetilde{M} \smallsetminus (q^*\Psi_{\cut})^{-1}(-\eps/2,0]$,
the blow-down map $q$ restricts to
an $\SS^1$-equivariant biholomorphic symplectomorphism  from
$(q^*\Psi_{\cut})^{-1}(-\infty,-\eps/2)$ to
$\Psi_{\cut}^{-1}(-\infty,-\eps/2)$ that intertwines  $\widetilde{\Psi}$
and $\Psi_{\cut}$. 
Moreover, 
\begin{equation}\label{disjoint}
\widetilde{\Psi}^{-1}(-\infty,-\eps/2)=
(q^*\Psi_{\cut})^{-1}(-\infty,-\eps/2)
\amalg
\big( \widetilde{\Psi}^{-1}(-\infty,-\eps/2)\cap(q^*\Psi_{\cut})^{-1}(-\eps/2,\infty)\big).
\end{equation}
Since $\widetilde{\Psi}$ is a proper moment map,
the preimage $\widetilde{\Psi}^{-1}(-\infty,-\eps/2)$ is connected.
Additionally, we chose $\eps$ so that $(q^*\Psi_{\cut})\inv(-\infty,-\eps/2) \neq \emptyset$.
Hence, \eqref{disjoint} implies that 
$$\widetilde{\Psi}^{-1}(-\infty,-\eps/2)=(q^*\Psi_{\cut})^{-1}(-\infty,-\eps/2) .$$
Together with the first paragraph, this proves Claim (2).

Finally, since 
 $\widetilde{\Psi}=q^*\Psi_{\cut}$  on 
$\widetilde{M} \smallsetminus (q^*\Psi_{\cut})\inv(-\eps/2,0]$, we have  inclusions 
$$\widetilde{\Psi}\inv(-\eps,\infty) \subset
(q^* \Psi_{\cut})\inv(-\eps,\infty) = (q^*\Psi_{\cut})\inv(-\eps,0] \subset q\inv(W).$$
Since $\widetilde{\omega}$ is tamed on $q\inv(W)$,
this proves Claim (3).
\end{proof}


\begin{remark}\label{r:toricfixed}
Let the circle  $\SS^1 \times \{1\}^{n-1} \subset (\SS^1)^n$
act on a symplectic toric manifold $(M,\omega,\Phi)$ with moment polytope
$\Delta = \Phi(M)$,
as described in Remark~\ref{r:toricreduction}, satisfying the assumptions of Proposition~\ref{p:prop}.
If $M$ is  cut  at sufficiently small $\eps > 0$,
then each of the
$k$ new fixed points in $M_{\cut}$ with
orbifold isotropy $\Z_2$  corresponds to a vertex
on the facet $\Delta \cap \big(\{\eps\} \times \R^{n-1}\big) \subset \Delta_{\cut}$;
(see Remark~\ref{r:toriccutting}).
If $M_{\cut}$  is  blown up at these fixed points by $2 \pi \eps$,
the resulting moment polytope $\widetilde{\Delta}$
agrees with $\Delta$ on $\{x \in \R^n \mid x_1 < 0 \}$,
has $k$ vertices in $\{0\} \times \R^{n-1}$, and a new facet $\widetilde{\Delta} \cap \big(\{\eps\} \times \R^{n-1}\big)$
that is the blow-up of $\Delta \cap \big(\{\eps\} \times \R^{n-1} \big)$  at  $k$ vertices; (see Remark~\ref{r:toricblowup}).
\end{remark}

\begin{remark}\label{r:reverse 2}
By reversing the action, as described in Remark~\ref{r:reverse action}, 
we see that 
Proposition~\ref{p:prop} still holds with the following modifications: 
Replace $(-\eps,0]$ by $[0,\eps)$ and $\{-2,1,\dots,1\}$ by $\{2,-1,\dots,-1\}$ in Claim (1); and replace each $(-\infty, -\eps/2)$ by $(\eps/2,\infty)$ in Claim (2).
\end{remark}

We conclude with the lemma that we used in the proof of Proposition~\ref{p:prop}.

\begin{lemma}\label{l:modifyforms}
Let $G$ be a closed subgroup of $\U(n)$.  
Let $U$ be a $G$-invariant neighbourhood of $0 \in \CC^n$, 
and let $\beta\in\Omega^2(U)^G$ be closed.  
Then there exists $\nu\in\Omega^1(U)^G$ with compact support such 
that $\widetilde{\beta}:=\beta-d\nu$ vanishes in a neighbourhood of $0$.
Finally, if $\beta\in\Omega^{1,1}(U)^G$, then we may also choose $\nu$ so that
 $\widetilde{\beta} \in \Omega^{1,1}(U)^G$.
\end{lemma}

\begin{proof}
By the Poincar\'e Lemma (and averaging), after possibly shrinking $U$ there exists $\mu \in \Omega^1(U)^G$ such that $d \mu = \beta.$  
Pick a $G$-invariant smooth function $\rho \colon U \to [0,1]$ with compact support that is $1$ on a neighbourhood of $0$. 
Let $ \nu= \rho\mu\in\Omega^1(U)^G$.
If $\beta\in\Omega^{1,1}(U)$, then by the Poincar\'e Lemma for $d$ and $\del$, there exists a smooth $G$-invariant potential
function $h\colon U\to\RR$ such that $\beta=\frac{i}{2}\del\delbar{h}$.  
Pick $\rho$ as before, and let $\nu := \frac{i}{4}\big(\delbar(\rho h)-\del(\rho h)\big)$.
\end{proof}

\section{Tame Birational Equivalence}\label{s:birational equiv}





In this section we study how, in our setting,  the  reduced space  
changes as we vary the value at which we reduce.
Our first goal is to prove the proposition below, which shows
that the birational equivalence theorem of Guillemin and Sternberg
also holds for tamed symplectic forms \cite{GS89} (if the fixed points are  isolated); 
see also \cite{DoHu} for the K\"ahler case.
We then use the fact that most of our tools work for arbitrary weights to prove Proposition~\ref{p:with crit pts orbifold}, which implies the tame analogue of (a special case of) a theorem of Godinho \cite{Go}.
This will allow us to analyse manifolds with fixed points that we construct
using Proposition~\ref{p:prop}.

\begin{proposition}\label{p:with crit pts}
Let $(M,J)$ be a complex manifold such that $\dim_\CC(M) > 1$
with a holomorphic $\CC^\times$-action, a symplectic form $\omega\in\Omega^2(M)^{\SS^1}$ tamed by the action, and a proper moment map $\Psi\colon M\to\RR$.  
Fix $a_-<0<a_+$ so that $\Psi^{-1}(a_-,a_+)$ contains exactly $k$ fixed points;
each lies in $\Psi^{-1}(0)$ and has weights $\{-1,1,\dots,1\}$.  Then there exists a complex orbifold $(X,I)$ and $\kappa,\eta\in H^2(X;\RR)$ so that:
\begin{itemize}
\item for all $s\in(a_-,0)$, the reduced space $M\red{s}\SS^1$ is biholomorphically symplectomorphic to $(X,I,\sigma_s)$,
 where $\sigma_s\in\Omega^2(X)$ satisfies $[\sigma_s]=\kappa- 2 \pi s \, \eta$; and
\item for all $s\in(0,a_+)$, the reduced space $M\red{s}\SS^1$ is biholomorphically symplectomorphic to $(\widehat{X},\widehat{I},\widehat{\sigma}_s)$, 
where $\widehat{\sigma}_s\in\Omega^2(\widehat{X})$ satisfies 
$$[\widehat{\sigma}_s]=q^*\kappa-2 \pi s \, q^*\eta- 2 \pi s\sum_{j=1}^k\mathcal{E}_j.$$
Here, $\widehat{X}$ is the blow-up of $X$ at smooth points $x_1,\dots,x_k\in X$, the map $q \colon \widehat{X}\to X$ is the blow-down map, and  $\mathcal{E}_j$ is the Poincar\'e dual of the exceptional divisor $q^{-1}(x_j)$ for all $j$.
\end{itemize}
Moreover, under the identifications above, 
the Euler class of the holomorphic principal $\C^\times$-bundle 
$\CC^\times \cdot \Psi\inv(s) \to M \red{s}\SS^1$ is $\eta$
for all $s \in (a_-,0)$,
and $q^* \eta + \sum \mathcal{E}_j$ for all $s \in (0,a_+)$.
\end{proposition}

This statement is particularly nice when there are no fixed points in $\Psi\inv(a_-,a_+)$.
(Note that the $\dim_\CC(M) = 1$ case is trivial.)

\begin{corollary}\label{c:locallyfree}
Let $(M,J)$ be a complex manifold with a holomorphic $\CC^\times$-action, a symplectic form $\omega\in\Omega^2(M)^{\SS^1}$ tamed by the action, and a proper moment map $\Psi\colon M\to\RR$.  
If $\Psi\inv(a_-,a_+)$ contains no fixed points, then
there exists a complex orbifold $(X,I)$ and $\kappa,\eta\in H^2(X)$, so that
the reduced space $M\red{s}\SS^1$ is biholomorphically symplectomorphic to $(X,I,\sigma_s)$,
where $\sigma_s\in\Omega^2(X)$ satisfies $[\sigma_s]=\kappa- 2 \pi s \, \eta$, for all $s \in (a_-,a_+)$.
Moreover, under the identifications above, 
the Euler class of the holomorphic principal $\C^\times$-bundle 
$\CC^\times \cdot \Psi\inv(s) \to M \red{s}\SS^1$ is $\eta$
for all $s \in (a_-,a_+)$.
\end{corollary}

By Proposition~\ref{p:reduction}, the reduced space $M \red{s} \SS^1$ is
isomorphic to $\big(\CC^\times \cdot \Psi\inv(s) \big)/\CC^\times$
for all regular $s \in \R$.
Hence, the first step in proving Proposition~\ref{p:with crit pts}
is analysing  $ \CC^\times \cdot \Psi\inv(s)$ for $s \in \R$,
or equivalently, $\Psi (\CC^\times \cdot x)$ for $x \in M$;
we accomplish this in the next two lemmas.

\begin{lemma}\label{l:limits}
Let $(M,J)$ be a complex manifold with a holomorphic $\CC^\times$-action, a symplectic form $\omega\in\Omega^2(M)^{\SS^1}$ tamed by the action, and a proper moment map $\Psi\colon M\to\RR$. 
Given $a \in \R$,  the following hold for all $x \in M$:
\begin{enumerate}
\item If $\Psi(x) < a$ then $a \in \Psi( \C^\times \cdot x) $  exactly if 
$$ \overline{\{ e^t \cdot x \mid  t \geq 0  \}} \cap \Psi\inv[\Psi(x),a] \cap M^{\SS^1} = \emptyset.$$
\item If $\Psi(x) > a$ then $a \in \Psi( \C^\times \cdot x) $  exactly if 
$$ \overline{\{ e^t \cdot x \mid  t \leq 0  \}} \cap \Psi\inv[a,\Psi(x)] \cap M^{\SS^1} = \emptyset.$$
\end{enumerate}
\end{lemma}

\begin{proof}
By symmetry, it suffices to consider $x \in M$ with  $\Psi(x) < a$.
If $x \in M^{\SS^1}$, then -- since the action is holomorphic -- $x$ is also fixed
by $\C^\times$.  Hence, $a \not\in \Psi(\C^\times \cdot x)$ and the equality displayed in Claim (1) does not hold.
Thus, we may assume that $x \not\in M^{\SS^1}$.

Assume first that
$$ p \in \overline{\{ e^t \cdot x \mid  t \geq 0  \}} \cap \Psi\inv[\Psi(x),a] \cap M^{\SS^1} \neq \emptyset.$$
Since $p \in M^{\SS^1}$,  
$e^t \cdot x \neq p$ for any $t \in \R$.
Hence, since  $p \in \overline{ \{ e^t \cdot x \mid t \geq 0\} }$, there  exists a sequence
of $t_i \in \R$ with $$\lim_{i \to \infty} t_i = \infty  \mbox{  and  }
\lim_{i \to \infty} e^{t_i} \cdot x = p.$$
Moreover, by 
Lemma~\ref{l:monotone},
the map $t \mapsto \Psi(e^t \cdot x)$ is strictly  increasing,
and so  $\Psi(e^t \cdot x) < \Psi(p) \leq a$ for all $t \in \R$. 
Therefore, since $\Psi$ is $\SS^1$-invariant, $a \not\in \Psi(\CC^\times \cdot x)$.

So assume instead that
$$\overline{\{ e^t \cdot x \mid  t \geq 0  \}} \cap \Psi\inv[\Psi(x),a] \cap M^{\SS^1} = \emptyset.$$
If $a \not\in \Psi(\CC^\times \cdot x)$, then since the map $t \mapsto  \Psi(e^t \cdot x)$ is increasing,
$ \Psi(e^t \cdot x) \in [\Psi(x),a)$ for all $t  \geq 0$.
Hence, since $\Psi$ is proper, the set  $K:=\overline{\{e^t\cdot x\mid t \geq 0 \}}$
 is a compact set that contains no fixed points.
Thus, Lemma~\ref{l:monotone} implies that there exists $\eps>0$ 
so that $\frac{d}{dt}\Big|_{t=0}\Psi(e^t\cdot y)>\eps$ for all $y\in K$.
  Since $e^s\cdot(e^t\cdot y)=e^{st}\cdot y$, this implies that $\frac{d}{dt}\Psi(e^t\cdot x)>\eps$ for all $t\geq 0$.
Thus, $\Psi(e^t\cdot x)>\Psi(x)+t\eps$ for all $t>0$.
Since this gives a contradiction, $a \in \Psi(\CC^\times \cdot x)$.
\end{proof}

\begin{lemma}\label{l:lnfplus}
Let $(M,J)$ be a complex manifold with a holomorphic $\CC^\times$-action, a symplectic form $\omega\in\Omega^2(M)^{\SS^1}$ tamed by the action, and a proper moment map $\Psi\colon M\to\RR$.  
Assume that that $M^{\SS^1} \cap \Psi^{-1}(a_-,a_+) = \{p_1,\dots,p_k\} \subset \Psi\inv(0)$.
There exists a $\CC^\times$-invariant neighbourhood
$W \subset M$ of $\{p_1,\dots,p_k\}$  which is $\CC^\times$-equivariantly biholomorphic
to a  neighbourhood of $\coprod_{j=1}^k \{0\}$ in
$\coprod_{j=1}^k \CC^{n^-_j} \times \CC^{n^+_j}$, 
where  $\C^\times$ acts linearly on each $\CC^{n^-_j}$ (respectively, $\CC^{n^+_j}$\!)
with negative (respectively, positive) weights. If we identify these neighbourhoods, then
$$\C^\times \cdot \Psi\inv(s) = 
\begin{cases}
U_- :=
\CC^\times \cdot \Psi\inv(a_-,a_+) \smallsetminus
\Big( \coprod_{j=1}^k \big( \{ 0\} \times \CC^{n_j^+}   \big)  \Big)
 & \mbox{if } s \in (a_-,0) \\
U_+ 
:= \CC^\times \cdot \Psi\inv(a_-,a_+) \smallsetminus
\Big( \coprod_{j=1}^k \big( \C^{n_j^-}  \times \{0\} \big)  \Big) 
& \mbox{if } s \in (0,a_+) \\
\left( U_- \cap U_+  \right) \cup \{p_1,\dots,p_k\} & \mbox{if }s = 0
\end{cases}
$$
where $ \coprod_{j=1}^k \big( \{ 0\} \times \CC^{n_j^+}   \big) \subset W  $ and
$ \coprod_{j=1}^k \big( \C^{n_j^-}  \times \{0\} \big) \subset W $.

\end{lemma}

\begin{proof}
The first claim is an immediate consequence of Proposition~\ref{p:local normal form}.

For each $j$, the connected component $W_j$ of $W$ containing $p_j$  
is open and $\C^\times$-invariant.
Hence, for any $x \in M$,
if $p_j \in \overline{ \{ e^t \cdot x \mid t \in \R \} } $
then $x \in W_j$
and so
\begin{equation*}
 p_j \in \overline{ \{ e^t \cdot x \mid t \geq 0 \} } 
 \text{ exactly if } x\in   \CC^{n_j^-}\times\{0\} \subset W_j. 
\end{equation*}
Since $p_1,\dots,p_k$ are the only fixed points in $\Psi\inv(a_-,a_+)$, 
this implies that
\begin{equation*}\label{positiveclosure}
\overline{ \{ e^t \cdot x \mid t \geq 0 \} }  \cap \Psi\inv(a_-,a_+) \cap M^{\SS^1} \neq \emptyset
 \text{ exactly if } x\in \coprod_{j=1}^k \big( \CC^{n_j^-}\times\{0\} \big) \subset  W.
\end{equation*}
By a similar argument, 
\begin{equation*}\label{negativeclosure}
\overline{ \{ e^t \cdot x \mid t \leq 0 \} }  \cap \Psi\inv(a_-,a_+) \cap M^{\SS^1} \neq \emptyset
 \text{ exactly if } x\in \coprod_{j=1}^k \big( \{0\} \times \CC^{n_j^+}\big) \subset  W. 
\end{equation*}
Therefore, since $M^{\SS^1} \cap \Psi\inv(a_-,a_+) = \{p_1,\dots,p_k\} \subset \Psi\inv(0)$,
the claim follows from  Lemma~\ref{l:limits}.
\end{proof}



If we remove the claims about the cohomology classes  $[\sigma_s]$ and 
$[\widehat{\sigma}_s]$ from  Proposition~\ref{p:with crit pts}, 
then the revised statement follows easily from 
Proposition~\ref{p:reduction} and Lemma~\ref{l:lnfplus} above.
(For details, see the proof
of Proposition~\ref{p:with crit pts} later in this section.)
Since these claims  
are formally identical
to statements in  \cite[Theorem 13.2]{GS89}, it  may seem that  we should
complete our proof by simply quoting their results.
Unfortunately,  the blow-down maps in our paper and their paper
are  not identical. 
Thus for completeness we include a proof of these claims,
which relies on the next two lemmas.

\begin{lemma}\label{l:restriction}
Let $\SS^1$ act on a symplectic manifold $(M,\omega)$ with moment map $\Psi \colon M \to \R$.
Assume that
$M^{\SS^1} = \{p_1,\dots,p_k\} \subset \Psi\inv(0)$.
Then there exists $\overline{\kappa} \in H^2\big(M / \SS^1\big)$ so that
$$[\omega_s] = i_s^*(\overline{\kappa}) - 2 \pi s \, \eta_s$$
for every regular value $s \in \R$,
where 
$\omega_s\in \Omega^2(M \red{s} \SS^1)$ is the reduced symplectic form,
$i_s \colon M \red{s} \SS^1 \to M/\SS^1$ is the natural inclusion,
and $\eta_s \in H^2(M \red{s} \SS^1)$ is the Euler class of the circle bundle
$\Psi\inv(s) \to M \red{s} \SS^1$.
\end{lemma}

\begin{proof}
Let $\theta \in \Omega^1\big(M \smallsetminus M^{\SS^1}\big)^{\SS^1}$ be a 
connection one-form, and let
$\Omega \in \Omega^2\big(\big(M \smallsetminus M^{\SS^1}\big)/\SS^1\big)$ be the 
associated curvature form.
Since  $\omega - d( \Psi \theta) \in \Omega^2\big(M \smallsetminus M^{\SS^1}\big)^{\SS^1}$
is closed and basic, it is the pull-back of a  closed form $\overline{\omega} \in \Omega^2\big(\big(M \smallsetminus M^{\SS^1}\big)/\SS^1\big) $.
Moreover, $i_s^*([\overline{\omega}]) = [\omega_s] + s [\, \Omega|_{M \red{s} \SS^1}]
= [\omega_s] + 2 \pi s \, \eta_s$
for all regular $s \in \R$.

By Lemma~\ref{l:modifyforms} and the  Bochner linearisation theorem
({\it cf.\ }Lemma~\ref{l:bochner}),
there exists $\nu \in \Omega^1(M)^{\SS^1}$ such that $\omega' := \omega - d \nu$
vanishes on a neighbourhood of $\{p_1,\dots,p_k\}$;
let $\Psi'  = \Psi  - \xi_M \hook \nu$.
Since 
$\xi_M \hook \omega' = - d \Psi'$
and $\Psi'(p_j) = \Psi(p_j) = 0$ for all $j$, the function $\Psi'$
also vanishes on a neighbourhood of $\{p_1,\dots,p_k\}$.
Therefore, $\omega' - d ( \Psi' \theta) \in \Omega^2(M)^{\SS^1}$ is  well-defined.
Hence, it is the pull-back
of a closed form $\overline{\omega}' \in \Omega^2(M/\SS^1)$; 
let $\overline{\kappa} = [\overline{\omega}'] \in H^2(M/\SS^1)$.
Since $\omega' - d(\Psi' \theta)$ and $\omega - d (\Psi \theta)$
differ by the differential of the basic form
$\nu - (\xi_M \hook \nu) \theta \in \Omega^1\big(M \smallsetminus M^{\SS^1}\big)^{\SS^1}$,
$i_s^*(\overline{\kappa})= i_s^*([\overline{\omega}])$ for all regular $s \in \R$.
\end{proof}

\begin{remark}
Alternately, we could
replace the second paragraph of the  proof above with
the following more sophisticated argument, which works for arbitrary $M^{\SS^1} \subset \Psi\inv(0)$:

In the de Rham model of equivariant cohomology, $\omega + \Psi$ represents an
equivariant cohomology class on $M$; moreover, 
$[\omega + \Psi]$ maps to $[\overline{\omega}]$ 
under the natural isomorphism
$H^*_{\SS^1}\big(M \smallsetminus M^{\SS^1}\big) \stackrel{\cong}{\to} H^*\big( \big(M \smallsetminus M^{\SS^1}\big)/\SS^1\big) $ 
\cite{AB}, \cite{GGK}.
Since $\Psi\big(M^{\SS^1}\big) = 0$,
the restriction  $[ \omega + \Psi] \big|_{M^{\SS^1}} \in H^2_{\SS^1}\big(M^{\SS^1}\big)$ 
is in the image of the natural inclusion $H^*\big(M^{\SS^1}\big) \hookrightarrow H^*_{\SS^1}\big(M^{\SS^1}\big)$.  Hence,
by the Leray spectral sequence  for the natural map $M \times_{\SS^1} E\SS^1 \to M/\SS^1$,
there exists a class $\overline{\kappa} \in H^*(M/\SS^1)$
that maps to   $[\omega + \Psi]$  under the 
pull-back  $H^*(M/\SS^1) \to H^*_{\SS^1}(M)$.
By construction,
$i_s^*(\overline{\kappa})= i_s^*([\overline{\omega}])$ for all regular $s \in \R$.
\end{remark}

\begin{lemma}\label{l:cont fct}
Let $(M,J)$ be a complex manifold with a holomorphic $\CC^\times$-action, a symplectic form $\omega\in\Omega^2(M)^{\SS^1}$ tamed by the action, and a proper moment map $\Psi\colon M\to\RR$.
Assume  that $M^{\SS^1} \cap \Psi\inv(a_-,a_+) = \{p_1,\dots,p_k\} \subset
\Psi\inv(0)$ for some $a_- < 0 < a_+$.
Then
for all $x \in \C^\times \cdot \Psi\inv(a_-,a_+)$ and $\tau \in [0,1]$, 
there exists a unique $G(x,\tau) \in M$ such that
$$G(x,\tau)\in\overline{\{e^t \cdot x \mid t \in \R \}}
\mbox{ and }
\Psi(G(x,\tau))= \tau \, \Psi(x);$$
moreover, $G \colon \C^\times \cdot \Psi^{-1}(a_-,a_+)\times[0,1]\to M$ 
is a continuous $\SS^1$-equivariant map.
\end{lemma}

\begin{proof}
Define 
$$\cU:=\{(x,s) \in \big(M\smallsetminus M^{\SS^1}\big) \times \R \mid
s \in \Psi(\CC^\times\cdot x)\}.$$  
By Lemma~\ref{l:monotone2},
 for all $(x, s) \in \cU$, there exists a unique $f(x,s) \in \RR$ such that $$\Psi(e^{f(x,s)}\cdot x)=s;$$
moreover, $\cU$ is open, and $f \colon \cU \to \RR$ is a smooth $\SS^1$-invariant function.
Moreover, Lemma~\ref{l:lnfplus} implies that, for
 $x \in \C^\times \cdot \Psi\inv(a_-,a_+) \smallsetminus \{p_1,\dots,p_k\}$ 
and $s \in (a_-,a_+)$,
$$ (x,s) \not\in \cU \Leftrightarrow 
\Big( x \in \coprod_{j=1}^k  \big( \C^{n_j^-} \times \{0\} \big)  \mbox{ and } s \geq 0 \Big)
\mbox{ or  }
\Big( x \in \coprod_{j=1}^k  \big( \{0\} \times \C^{n_j^+}  \big)  \mbox{ and } s \leq 0 \Big).
$$

Therefore, for all $x \in \C^\times \cdot \Psi\inv(a_-,a_+)$ and $\tau \in [0,1]$, there exists
a unique $G(x,\tau) \in M$ such that
$G(x,\tau) \in \overline{ \{e^t \cdot x \mid t \in \R \}}$ and
$\Psi(G(x,\tau)) = \tau \,\Psi(x)$:
$$
G(x,\tau) =
\begin{cases}
p_j & \mbox{if } x = p_j \mbox{ for some }j,\\
p_j & \mbox{if } 
x \in \big( \big(\C^{n_j^-} \times \{0\} \big) 
\cup \big(\{0\} \times \C^{n_j^+}\big) \big)
\mbox{ for some } j
\mbox{ and } 
\tau = 0 
, \\
e^{f(x,\tau \, \Psi(x))} \cdot x& \mbox{otherwise.}
\end{cases}
$$
Clearly, $G \colon \C^\times \cdot \Psi\inv(a_-,a_+) \times [0,1] \to M$ is an $\SS^1$-equivariant map
that is continuous on the complement of $G\inv(\{p_1,\dots,p_k\})$.
Let $U \subseteq \Psi\inv(a_-,a_+)$ be an $\SS^1$-invariant neighbourhood of $p_j$.
By Corollary~\ref{delta}, 
there exists $\delta>0$ and an $\SS^1$-invariant open neighbourhood $V\subseteq U$ of $p_j$ 
such that $\CC^\times \cdot V \cap \Psi\inv(-\delta,\delta) \subset V$.
Then  
$$\big\{ (x,\tau) \in (\C^\times \cdot V) \times [0,1]\, \big|\, 
| \tau \, \Psi(x)| < \delta \big\}$$
is an open subset of $G\inv(U)$ that contains
$G\inv(p_j)$.
Therefore, $G$ is continuous.
\end{proof}

This has an immediate corollary.

\begin{corollary}\label{c:cont fct}
Assume that we are in the situation of Lemma~\ref{l:cont fct}.
Fix  a non-zero $s \in (a_-,a_+)$, and let $U_s := \C^\times \cdot \Psi\inv(s)$.
For every $[x] \in U_s/\CC^\times$, there exists a  unique 
$\SS^1$-orbit   
$$h_s([x]) \in 
\Psi\inv(0)/\SS^1
 \cap 
\big(\overline{\CC^\times \cdot x} \big)/\SS^1;
$$
moreover, 
$h_s \colon  U_s/\CC^\times \to \Psi\inv(a_-,a_+)/\SS^1$ is a continuous map.
Additionally, if
$j_s \colon M \red{s} \SS^1 \to U_s/\CC^\times$ 
is induced by the inclusion map,
then the composition $h_s \circ j_s \colon M \red{s} \SS^1 \to \Psi\inv(a_-,a_+)/\SS^1$ is 
isotopic to the inclusion map.
\end{corollary}

We now prove our main proposition.

\begin{proof}[Proof of Proposition~\ref{p:with crit pts}]
We may assume without loss of generality that $0 \in (a_-,a_+)$ (this is by hypothesis if there are fixed points).
By Lemma~\ref{l:lnfplus}, there exists a $\CC^\times$-invariant
neighbourhood  $W$ of $\Psi\inv(0) \cap M^{\SS^1}$ which
is $\CC^\times$-equivariantly biholomorphic to a neighbourhood of $\coprod_{j=1}^k\{0\}$ in 
$\coprod_{j=1}^k\CC^n$, where in each case $\CC^\times$ acts on $\CC^n$ by 
\begin{equation}\label{action}
\lambda \cdot (z_1,\dots,z_n) = (\lambda^{-1} z_1, \lambda z_2, \dots, \lambda z_n)
\end{equation}
for all $\lambda \in \CC^\times $ and $z \in \CC^n$.
Moreover, if we identify these neighbourhoods then
$$\C^\times \cdot \Psi\inv(s) = 
\begin{cases}
U_- :=
\CC^\times \cdot \Psi\inv(a_-,a_+) \smallsetminus
\Big( \coprod_{j=1}^k \big( \{ 0\} \times \CC^{n-1}   \big)  \Big)
 & \mbox{if } s \in (a_-,0) \\
U_+ 
:= \CC^\times \cdot \Psi\inv(a_-,a_+) \smallsetminus
\Big( \coprod_{j=1}^k \big( \CC \times \{ 0\}   \big)  \Big)
& \mbox{if } s \in (0,a_+). \\
\end{cases}
$$
By Proposition~\ref{p:reduction}, 
the quotients $U_\pm/\C^\times$ are naturally complex orbifolds,
and $U_\pm$ is a holomorphic $\C^\times$-bundle over $U_\pm/\C^\times$
with Euler class $\eta_\pm \in H^2(U_\pm/\C^\times)$.
Moreover, the inclusion $\Psi\inv(s_\pm) \hookrightarrow U_\pm$ induces
a biholomorphism $j_{s_\pm} \colon M \red{s_\pm} \SS^1 \to U_\pm/\C^\times$ for each $s_- \in (a_-,0)$
and $s_+ \in (0,a_+)$.
Thus, 
the Euler class
of the $\SS^1$-bundle
$\Psi\inv(s_\pm) \to M \red{s_\pm} \SS^1$ 
is $j_{s_\pm}^*( \eta_\pm)$ 
for each $s_- \in (a_-,0)$
and $s_+ \in (0,a_+)$.
By Corollary~\ref{c:cont fct}, 
there exist continuous maps $h_\pm \colon U_\pm/\C^\times
\to \Psi\inv(a_-,a_+)/\SS^1$ such that 
$h_\pm([x]) \in
\Psi\inv(0) /\SS^1 \cap \big(\overline{\C^\times \cdot x}  
\big)/\SS^1$ for all $[x] \in U_\pm/\C^\times$.
Moreover, the composition $h_\pm \circ j_{s_\pm}\colon M \red{s_\pm} \SS^1 \to \Psi^{-1}(a_-,a_+)/\SS^1$ is isotopic to the
the natural inclusion
for each $s_- \in (a_-,0)$
and $s_+ \in (0,a_+)$.
Therefore, applying Lemma~\ref{l:restriction} to $\Psi\inv(a_-,a_+)$,
there exists $\overline{\kappa} \in H^2(\Psi\inv(a_-,a_+)/\SS^1)$
so that
$$[\omega_{s_\pm}] = 
j_{s_\pm}^*\big( h_\pm^*(\overline{\kappa}) - 2 \pi  s_\pm \, \eta_\pm \big)$$
for each $s_- \in (a_-,0)$ and $s_+ \in (0,a_+)$, where $\omega_{s_\pm} \in \Omega^2(M \red{s_\pm} \SS^1)$
is the reduced symplectic form.

The quotient
$\CC\times_{\CC^\times}(\CC^{n-1}\smallsetminus\{0\})$ is the
blow-up of
$ \CC^{n-1} \cong (\CC \smallsetminus \{0\}) \times_{\CC^\times}\CC^{n-1}$ at $0$;
the blow-down map sends $[u;z_1,\dots,z_{n-1}]$ to 
$ [1; u z_1, \dots, u z_{n-1}]$.
Therefore, 
$U_{+}/\CC^\times$ is the blow-up of 
$U_{-}/\CC^\times$ at the  smooth points
$$\{x_1,\dots,x_k\} := \coprod_j \big(( \CC \smallsetminus \{0\}) \times \{ 0\}\big)/\CC^\times \subset W/\CC^\times;$$
moreover, $h_- \circ q = h_+$, where
 $q \colon U_+/ \CC^\times \to U_-/\CC^\times$ is the blow-down map.
(To see this, note that  
$h_\pm([x]) = [0] \in \C^n/\SS^1$
for all non-zero
$x \in (\C  \times \{0\}) \cup (\{0\} \times  \C^{n-1}) \subset \C^n \subset W$.
 Hence, $h_+^*(\overline\kappa) = q^*(h_-^*(\overline\kappa))$.

Finally, we complete the proof by using an argument from \cite{GS89} to show that
\begin{equation}\label{etaplus}
\eta_+ = q^*(\eta_-)  + \sum_j  \mathcal{E}_j,
\end{equation}
where $\mathcal{E}_j$ is the Poincar\'e dual of the exceptional divisor 
$E_j := q\inv(x_j)$ for each $j$.
Define a section of the $\C^\times$-bundle
$$\big(\C \smallsetminus \{0\} \big) \times \big( \C^{n-1} \smallsetminus \{0\} \big) \to
 \big(\C \smallsetminus \{0\} \big) \times_{\C^\times} \big( \C^{n-1} \smallsetminus \{0\} \big) $$
 by $\sigma[u; z_1,\dots,z_{n-1}] = (1;u z_1,\dots,u z_{n-1})$.
Let $L$ be the  $\C^\times$-bundle  over $U_+/\C^\times$  
constructed  by using $\sigma$ to glue together the $\C^\times$-bundle
$W \cap U_+ \to (W \cap U_+)/\C^\times$ and the trivial $\C^\times$-bundle  over
$U_+/\CC^\times \smallsetminus \coprod_j E_j$.
On the one hand,  $\sigma$ does not extend to a section of the
$\C^\times$-bundle
$$\C  \times \big( \C^{n-1} \smallsetminus \{0\} \big) \to
\C  \times_{\C^\times} \big( \C^{n-1} \smallsetminus \{0\} \big)$$ associated to $W \cap U_+$; instead,
it extends to a holomorphic section of the associated $\C$-bundle that is transverse to the
zero section, and vanishes  
exactly on the exceptional divisor. 
Hence, the Euler class of $L$ is $\sum_j \mathcal{E}_j$.
On the other hand, $\sigma$ does extend to a section of the $\C^\times$-bundle
$$\big(\C \smallsetminus \{0\} \big) \times  \C^{n-1}  \to
\big(\C \smallsetminus \{0\} \big) \times_{\C^\times} \C^{n-1}$$
associated to $W \cap U_-$.
Since the Euler class of the tensor product of line bundles is the sum of
their Euler classes, \eqref{etaplus} follows immediately.
\end{proof}

Finally, we generalise Proposition~\ref{p:with crit pts} so that it applies to
manifolds
constructed using Proposition~\ref{p:prop}.

\begin{proposition}\label{p:with crit pts orbifold}
Let $(M,J)$ be a complex manifold such that $\dim_\C M > 1$
with a holomorphic $\CC^\times$-action, a symplectic form $\omega\in\Omega^2(M)^{\SS^1}$ tamed by the action, and a proper moment map $\Psi\colon M\to\RR$.  
Fix $a_-<0<a_+$ so that $\Psi^{-1}(a_-,a_+)$ contains exactly $k$ fixed points; each lies in $\Psi^{-1}(0)$ and has weights $\{-2,1,\dots,1\}$. 
Then there exists a complex orbifold $(X,I)$ and classes $\kappa,\eta\in H^2(X;\RR)$ so that 
\begin{itemize}
\item for all $s\in(a_-,0)$, the reduced space $M\red{s}\SS^1$ is biholomorphically symplectomorphic to $(X,I,\sigma_s)$,
 where $\sigma_s\in\Omega^2(X)$ satisfies $[\sigma_s]=\kappa-2 \pi s \, \eta$; and
\item for all $s\in(0,a_+)$, the reduced space $M\red{s}\SS^1$ is biholomorphically symplectomorphic to $(\widehat{X},\widehat{I},\widehat{\sigma}_s)$, 
where $\widehat{\sigma}_s\in\Omega^2(\widehat{X})$ satisfies $$[\widehat{\sigma}_s]=q^*\kappa-2 \pi s \, q^*\eta- \pi s\sum_{j=1}^k\mathcal{E}_j.$$
Here, $\widehat{X}$ is the blow-up of $X$ at isolated $\Z_2$-singularities 
$x_1,\dots,x_k\in X$, the map $q\colon\widehat{X}\to X$ is the blow-down map, and 
 $\mathcal{E}_j$ is the Poincar\'e dual of the exceptional divisor $q^{-1}(x_j)$ for all $j$.
\end{itemize}
Moreover, under the identifications above, 
the Euler class of the holomorphic principal $\C^\times$-bundle 
$\CC^\times \cdot \Psi\inv(s) \to M \red{s}\SS^1$ is $\eta$
for all $s \in (a_-,0)$,
and $q^* \eta + \sum \mathcal{E}_j/2$ for all $s \in (0,a_+)$.
\end{proposition}

\begin{proof}
The first paragraph of the proof of Proposition~\ref{p:with crit pts}
applies without change, except that  \eqref{action} should be replaced by
$$\lambda \cdot (z_1,\dots,z_n) = (\lambda^{-2} z_1, \lambda z_2, \dots, \lambda z_n).$$
By Definition~\ref{d:blow-up}, $U_+/\CC^\times$ 
is the blow-up of $U_-/\C^\times$ at the  isolated $\Z_2$-singularities 
$$\{x_1,\dots,x_k\} := \coprod_j \big(( \CC \smallsetminus \{0\}) \times \{ 0\}\big)/\CC^\times;$$
moreover, 
$h_- \circ q = h_+$, 
where $q \colon U_+/\CC^\times \to U_-/\CC^\times$ is the blow-down map.

In analogy with the proof of Proposition~\ref{p:with crit pts}, this reduces the argument to showing that
\begin{equation}\label{e:with crit pts orbifold}
2 \eta_+ = 2 q^*(\eta_-) + \sum_j \mathcal{E}_j,
\end{equation}
where $\mathcal{E}_j$ is the Poincar\'e dual of the exceptional divisor $E_j :=
q\inv(x_j)$.
Define a section of the $\CC^\times$-bundle $$\big(\C \smallsetminus \{0\}\big) \times_{\Z_2} \big(\C^{n-1} \smallsetminus \{0\} \big)
\to \big(\C \smallsetminus \{0\}\big) \times_{\C^\times} \big(\C^{n-1} \smallsetminus \{0\}\big)$$
by $\sigma[u; z_1,\dots,z_{n-1}] = [1; \sqrt{u} z_1, \dots, \sqrt{u} z_{n-1}]$, and construct a $\CC^\times$-bundle $L$ over $U_+/\CC^\times$ by using $\sigma$ to glue together the $\C^\times$-bundle $(W \cap U_+)/\Z_2 \to (W \cap U_+)/\C^\times$ and the trivial bundle over $U_+/\C^\times \smallsetminus \coprod_j E_j$. 
The Euler class of $L$ is $\sum_j \mathcal{E}_j$ 
because  $\sigma$ extends to a holomorphic section of the $\CC$-bundle associated
to the $\C^\times$-bundle 
$$\C \times_{\Z_2}\big( \C^{n-1} \smallsetminus \{0\}\big) \to \C \times_{\C^\times} \big(\C^{n-1} \smallsetminus \{0\}\big)$$ that is transverse to the zero section and vanishes exactly on the exceptional divisor. 
Since  $\sigma$ extends to a section of the $\CC^\times$-bundle $$\big(\C \smallsetminus \{0\} \big)\times_{\Z_2} \C^{n-1} \to \big(\C \smallsetminus \{0\}\big) \times_{\C^\times} \C^{n-1},$$ and
since the Euler class of $U_\pm/\Z^2 \to U_\pm/\C^\times$ is twice the Euler class  of $U_\pm \to U_\pm/\C^\times$,
this proves \eqref{e:with crit pts orbifold}.

\end{proof}

\begin{remark}\label{r:toricbirational}
Let the circle  $\SS^1 \times \{1\}^{n-1} \subset (\SS^1)^n$
act on a symplectic toric manifold $(M,\omega,\Phi)$ with moment polytope
$\Delta$,
as described in Remark~\ref{r:toricreduction}, satisfying the assumptions of Proposition~\ref{p:with crit pts} (or Proposition~\ref{p:with crit pts orbifold}).
It is straightforward to see that, as we vary $s_- \in (a_-,0)$,
the  moment polytopes
$\Delta \cap \big(\{s_-\} \times \R^{n-1} \big)$
of the reduced spaces 
$M\red{s_-}\SS^1$ all have the same facets, but the position
of these facets varies linearly in $s_-$.  The identical process occurs as
we vary the moment polytopes of the reduced spaces over $s_+ \in (0,a_+)$.
However, the latter polytopes have $k$ extra facets, corresponding
to the blow-up  at the $k$ fixed points.
 (See Remarks~\ref{r:toricreduction} and \ref{r:toricblowup}).
\end{remark}

\begin{remark}\label{r:reverse 3}
By reversing the
action, as described in Remark~\ref{r:reverse action}, we see that
Proposition~\ref{p:with crit pts} (respectively, Proposition~\ref{p:with crit pts orbifold}) 
still holds with the following modifications:
Replace $\{-1,1,\dots,1\}$ by $\{1,-1,\dots,-1\}$ 
(respectively,  replace $\{-2,1,\dots,1\}$ by $\{2,-1,\dots,-1\}$);  
replace each $(0,a_+)$ by $(a_-,0)$  and each $(a_-,0)$ by $(0,a_+)$; and
replace each $\sum_j \mathcal{E}_j$ by 
$- \sum_j \mathcal{E}_j$.
\end{remark}


\end{document}